\documentclass[12pt,a4paper,reqno,oneside]{amsart}
 
\usepackage[utf8]{inputenc}
\usepackage[T1]{fontenc}
\usepackage{lmodern}
\usepackage{amsmath}
\usepackage{amssymb}
\usepackage{amsthm}
\allowdisplaybreaks
\usepackage{geometry}
\usepackage{graphicx}
\usepackage{bbm}
\usepackage{float}
\usepackage{enumerate}
\usepackage{scrhack}
\usepackage{cleveref}

\newtheoremstyle{theorem}	
   {}       			
   {}      				
   {\itshape}  				
   {\parindent} 			
   {\bfseries} 				
   {}         				
   {.7em}      				
   {}          				

\newtheoremstyle{definition}
  {}       				
  {}      				
  {} 		 				
  {\parindent} 				
  {\bfseries} 				
  {}         				
  {.7em}      				
  {}          				
  
\newtheoremstyle{remark}
  {}       			
  {}      			
  {} 					
  {\parindent} 			
  {\itshape} 			
  {.}         			
  {.7em}      			
  {}          			

\theoremstyle{theorem}
\newtheorem{theorem}{Theorem}[section]
\newtheorem{corollary}[theorem]{Corollary}
\newtheorem{lemma}[theorem]{Lemma}
\newtheorem{proposition}[theorem]{Proposition}

\theoremstyle{remark}
\newtheorem{remark}[theorem]{Remark}

\theoremstyle{definition}

\newtheorem{definition}[theorem]{Definition}

\newcommand{\Rbb}{\mathbb{R}}
\newcommand{\Nbb}{\mathbb{N}}
\newcommand{\Pbb}{\mathbb{P}}
\newcommand{\Fcal}{\mathcal{F}}
\newcommand{\Ecal}{\mathcal{E}}
\newcommand{\Ebb}{\mathbb{E}}

\newcommand{\Ccal}{\mathcal{C}}
\newcommand{\Fbb}{\mathbb{F}}

\newcommand{\Scal}{\mathcal{S}}
\newcommand{\Qbb}{\mathbb{Q}}

\newcommand{\btilde}{\tilde{b}}
\newcommand{\bhat}{\hat{b}}
\newcommand{\Wcal}{\mathcal{W}}

\newcommand{\Pcal}{\mathcal{P}}
\newcommand{\Bcal}{\mathcal{B}}
\newcommand{\Lip}{\text{Lip}}
\newcommand{\Kcal}{\mathcal{K}}
\newcommand{\Acal}{\mathcal{A}}
\newcommand{\Gcal}{\mathcal{G}}

\newcommand{\Gbb}{\mathbb{G}}
\newcommand{\bfr}{\mathfrak{b}}
\DeclareMathOperator*{\esssup}{ess\,sup}

\usepackage{color}
\definecolor{darkgreen}{rgb}{0, .5, 0}
\definecolor{darkred}{rgb}{.5, 0, 0}

\begin{document}


\title[Strong Solutions of MFSDEs with Irregular Drift]{Strong Solutions of Mean-Field Stochastic Differential Equations with Irregular Drift}
\author[M.Bauer]{Martin Bauer}
\address{M. Bauer: Department of Mathematics, LMU, Theresienstr. 39, D-80333 Munich, Germany.}
\email{bauer@math.lmu.de} 
\author[T. Meyer-Brandis]{Thilo Meyer-Brandis}
\address{T. Meyer-Brandis: Department of Mathematics, LMU, Theresienstr. 39, D-80333 Munich, Germany.}
\email{meyerbra@math.lmu.de} 
\author[F. Proske]{Frank Proske}
\address{F. Proske: CMA, Department of Mathematics, University of Oslo, Moltke Moes vei 35, P.O. Box 1053 Blindern, 0316 Oslo, Norway.}
\email{proske@math.uio.no} 
\date{\today}
\maketitle


\textbf{Abstract.} We investigate existence and uniqueness of strong solutions of mean-field stochastic differential equations with irregular drift coefficients. Our direct construction of strong solutions is mainly based on a compactness criterion employing Malliavin Calculus together with some local time calculus. Furthermore, we establish regularity properties of the solutions such as Malliavin differentiablility as well as Sobolev differentiability in the initial condition. Using this properties we formulate an extension of the Bismut-Elworthy-Li formula to mean-field stochastic differential equations to get a probabilistic representation of the first order derivative of an expectation functional with respect to the initial condition.\\[0.5cm]

\section{Introduction}

Throughout this paper, let $T>0$ be a given time horizon. 
Mean-field stochastic differential equations (hereafter mean-field SDE), also referred to as McKean-Vlasov equations, given by
\begin{align}\label{eq:MFSDEgeneral}
	dX_t^x = b(t,X_t^x,\Pbb_{X_t^x}) dt + \sigma(t,X_t^x,\Pbb_{X_t^x}) dB_t, \quad X_0^x = x \in \Rbb^d, \quad t\in [0,T],
\end{align}
are an extension of stochastic differential equations where the coefficients are allowed to depend on the law of the solution in addition to the dependence on the solution itself. Here $b: \Rbb^+ \times \Rbb^d \times \Pcal_1(\Rbb^d) \to \Rbb^d$ and $\sigma: \Rbb^+ \times \Rbb^d \times \Pcal_1(\Rbb^d) \to \Rbb^{d\times n}$ are some given drift and volatility coefficients, $(B_t)_{t\in[0,T]}$ is an $n$-dimensional Brownian motion,
\begin{align*}
		\Pcal_1(\Rbb^d) := \left\lbrace \mu \left| \mu \text{ probability measure on } (\Rbb^d, \Bcal(\Rbb^d)) \text{ with } \int_{\Rbb^d} |x| d\mu(x) < \infty \right.\right\rbrace
	\end{align*}
is the space of probability measures over $(\Rbb^d,\Bcal(\Rbb^d))$ with existing first moment, and $\Pbb_{X_t^x}$ is the law of $X_t^x$ with respect to the underlying probability measure $\Pbb$. Based on the works of Vlasov \cite{Vlasov_VibrationalPropertiesofElectronGas}, Kac \cite{Kac_FoundationsOfKineticTheory} and McKean \cite{McKean_AClassOfMarkovProcess}, mean-field SDEs arised from Boltzmann's equation in physics, 
	which is used to model weak interaction between particles in a multi-particle system. 
Since then the study of mean-field SDEs has evolved as an active research field with numerous applications. Various extensions of the class of mean-field SDEs as for example replacing the driving noise by a Lévy process or considering backward equations have been examined e.g.~in \cite{JourdainMeleardWojbor_NonlinearSDEs}, \cite{BuckdahnDjehicheLiPeng_MFBSDELimitApproach}, \cite{BuckdahnLiPeng_MFBSDEandRelatedPDEs}, and \cite{BuckdahnLiPengRainer_MFSDEandAssociatedPDE}. With their work on mean-field games in \cite{LasryLions_MeanFieldGames}, Lasry and Lions have set a cornerstone in the application of mean-field SDEs in Economics and Finance, see also \cite{Cardaliaguet_NotesOnMeanFieldGames} for a readily accessible summary of Lions' lectures at Collège de France. As opposed to the analytic approach taken in \cite{LasryLions_MeanFieldGames}, Carmona and Delarue developed a probabilistic approach to mean-field games, see e.g. \cite{CarmonaDelarue_ProbabilisticAnalysisofMFG}, \cite{CarmonaDelarue_MasterEquation}, \cite{CarmonaDelarue_FBSDEsandControlledMKV}, \cite{CarmonaDelarueLachapelle_ControlofMKVvsMFG} and \cite{CarmonaLacker_ProbabilisticWeakFormulationofMFGandApplications}. More recently, the mean-field approach also found application in systemic risk modeling, especially in models for inter-bank lending and borrowing, see e.g. \cite{CarmonaFouqueMousaviSun_SystemicRiskandStochasticGameswithDelay}, \cite{CarmonaFouqueSun_MFGandSystemicRisk}, \cite{FouqueIchiba_StabilityinaModelofInterbankLending}, \cite{FouqueSun_SystemicRiskIllustrated}, \cite{GarnierPapanicolaouYang_LargeDeviationsforMFModelofSystemicRisk}, \cite{KleyKlueppelbergReichel_SystemicRiskTroughContagioninaCorePeripheryStructuredBankingNetwork}, and the cited sources therein.\par
	
In this paper we study existence, uniqueness and regularity properties of (strong) solutions of one-dimensional mean-field SDEs of the type 
	\begin{align}\label{introdMFSDE}
		dX_t^x = b(t,X_t^x, \Pbb_{X_t^x}) dt + dB_t, \quad X_0^x = x \in \Rbb, \quad t\in[0,T].
	\end{align}
If the drift coefficient $b$ is of at most linear growth and Lipschitz continuous, existence and uniquenss of (strong) solutions of \eqref{introdMFSDE} are well understood. Under further smoothness assumptions on $b$, differentiability in the initial condition $x$ and the relation to non-linear PDE's is studied in \cite{BuckdahnLiPengRainer_MFSDEandAssociatedPDE}. We here consider the situation when the drift $b$ is allowed to be irregular. More precisely, in addition to some linear growth condition we basically only require measurability in the second variable and some continuity in the third variable.

The first main contribution of this paper is to establish existence and uniqueness of strong solutions of mean-field SDE \eqref{introdMFSDE} under such irregularity assumptions on $b$. To this end, we firstly consider existence and uniqueness of weak solutions of mean-field SDE \eqref{introdMFSDE}. In \cite{Chiang_MKVWithDiscontinuousCoefficients}, Chiang proves the existence of weak solutions for time-homogeneous mean-field SDEs with drift coefficients that are of linear growth and allow for certain discontinuities. 
In the time-inhomogeneous case, Mishura and Veretennikov ensure in \cite{MishuraVeretennikov_SolutionsOfMKV} the existence of weak solutions by requiring in addition to linear growth that the drift is of the form 
\begin{equation}\label{eq:DriftVer}
b(t,y,\mu) = \int \overline{b}(t,y,z) \mu(dz),
\end{equation} 
for some $\overline{b}: [0,T] \times \Rbb \times \Rbb \rightarrow \Rbb$. In \cite{LiMin_WeakSolutions}, Li and Min show the existence of weak solutions of mean-field SDEs with path-dependent coefficients, supposing that the drift is bounded and continuous in the third variable. We here relax the boundedness requirement in \cite{LiMin_WeakSolutions} (for the non-path-dependent case) and show existence of a weak solution of  \eqref{introdMFSDE} by merely requiring that $b$ is continuous in the third variable, i.e. for all $\mu \in \Pcal_1(\Rbb)$ and all $\varepsilon>0$ exists a $\delta>0$ such that
	\begin{align}\label{continuityThirdVariable}
		\left( \forall \nu\in\Pcal_1(\Rbb): \Kcal(\mu,\nu)<\delta \right) \Rightarrow |b(t,y,\mu)-b(t,y,\nu)| < \varepsilon, \quad t\in[0,T], \quad y\in\Rbb,
	\end{align}
	and of at most linear growth, i.e. there exists a constant $C>0$ such that for all $t\in[0,T]$, $y\in \Rbb$ and $\mu \in \Pcal_1(\Rbb)$,
	\begin{align}\label{linearGrowth} 
		\vert b(t,y,\mu) \vert &\leq C(1+|y|+\Kcal(\mu,\delta_0)).
	\end{align}
	Here $\delta_0$ is the Dirac-measure in 0 and $\Kcal$ the Kantorovich metric:
	\begin{align*}
		\Kcal(\lambda,\nu) &:= \sup_{h\in\Lip_1(\Rbb)} \left\vert \int_{\Rbb} h(x) (\lambda-\nu)(dx) \right\vert, \quad \lambda,\nu \in \Pcal_1(\Rbb),
	\end{align*}
	where $\Lip_1(\Rbb)$ is the space of Lipschitz continuous functions with Lipschitz constant 1 (for an explicit definition see the notations below).	 
Further we show that if $b$ admits a modulus of continuity in the third variable (see \Cref{modulusOfContinuity}) in addition to \eqref{continuityThirdVariable} and  \eqref{linearGrowth}, then there is weak uniqueness (or uniqueness in law) of solutions of \eqref{introdMFSDE}.


In order to establish the existence of strong solutions of \eqref{introdMFSDE}, we then show that any weak solution actually is a strong solution. Indeed, given a weak solution $X^x$ (and in particular its law) of mean-field SDE \eqref{introdMFSDE}, one can re-interprete $X$ as the solution of a common SDE
	\begin{align}\label{commonSDE}
		dX_t^x = b^{\Pbb_{X}}(t,X_t^x) dt + dB_t, \quad X_0^x = x \in \Rbb, \quad t\in [0,T],
	\end{align}
where $b^{\Pbb_{X}}(t,y) := b(t,y, \Pbb_{X_t^x})$. This re-interpretation allows to apply the ideas and techniques developed in \cite{MeyerBrandisBanosDuedahlProske_ComputingDeltas},\cite{MenoukeuMeyerBrandisNilssenProskeZhang_VariationalApproachToTheConstructionOfStrongSolutions} and \cite{MeyerBrandisProske_ConstructionOfStrongSolutionsOfSDEs} on strong solutions of SDEs with irregular coefficients to equation \eqref{commonSDE}. In order to deploy these results and to prove that the weak solution $X^x$ is indeed a strong solution, we still assume condition \eqref{continuityThirdVariable}, i.e.~the drift coefficient $b$ is supposed to be continuous in the third variable, but require the following particular form proposed in  \cite{MeyerBrandisBanosDuedahlProske_ComputingDeltas} of the linear growth condition \eqref{linearGrowth}:
	\begin{align}\label{formDrift}
		b(t,y,\mu) = \bhat(t,y,\mu) + \btilde(t,y,\mu),
	\end{align}
	where $\bhat$ is merely measurable and bounded and $\btilde$ is of at most linear growth \eqref{linearGrowth} and Lipschitz continuous in the second variable, i.e. there exists a constant $C>0$ such that for all $t\in[0,T]$, $y_1,y_2 \in \Rbb$ and $\mu \in \Pcal_1(\Rbb)$,
	\begin{align}\label{lipschitzContinuous}
		\vert  \btilde(t,y_1,\mu) -  \btilde (t,y_2,\mu) \vert &\leq C|y_1-y_2|.
	\end{align}
We remark that while a typical approach to show existence of strong solutions is to establish existence of weak solutions together with pathwise uniqueness (Yamada-Watanabe Theorem), in \cite{MeyerBrandisBanosDuedahlProske_ComputingDeltas},\cite{MenoukeuMeyerBrandisNilssenProskeZhang_VariationalApproachToTheConstructionOfStrongSolutions} and \cite{MeyerBrandisProske_ConstructionOfStrongSolutionsOfSDEs} the existence of strong solutions is shown by a direct constructive approach based on some compactness criterion employing Malliavin calcuclus. Further, pathwise (or strong) uniqueness is then a consequence of weak uniqueness. We also remark that in \cite{MishuraVeretennikov_SolutionsOfMKV} the existence of strong solutions of mean-field SDEs is shown in the case that the drift is of the special form \eqref{eq:DriftVer} where $\overline{b}$ fulfills certain linear growth and Lipschitz conditions.

The second contribution of this paper is the study of certain regularity properties of strong solutions of mean-field equation \eqref{introdMFSDE}. Firstly, from the constructive approach to strong solutions based on \cite{MeyerBrandisBanosDuedahlProske_ComputingDeltas}, \cite{MenoukeuMeyerBrandisNilssenProskeZhang_VariationalApproachToTheConstructionOfStrongSolutions} and \cite{MeyerBrandisProske_ConstructionOfStrongSolutionsOfSDEs} we directly gain Malliavin differentiability of strong solutions of SDE~\eqref{commonSDE}, i.e.~Malliavin differentiability of strong solutions of mean-field SDE \eqref{introdMFSDE}. 
Similar to \cite{MeyerBrandisBanosDuedahlProske_ComputingDeltas} we provide a probabilistic representation of the Malliavin derivative using the local time-space integral introduced in \cite{Eisenbaum_LocalTimeIntegral}.

Secondly, we investigate the regularity of the dependence of a solution $X^x$ on its initial condition x. For the special case where the mean-field dependence is given via an expectation functional of the form
\begin{align}\label{specialMFSDE}
		dX_t^x = \overline{b}(t,X_t^x, \Ebb[\varphi(X_t^x)]) dt + dB_t, \quad X_0^x = x \in \Rbb, \quad t\in[0,T],
\end{align}
for some $\overline{b}: [0,T] \times \Rbb \times \Rbb \rightarrow \Rbb$, continuous differentiability of $X^x$ with respect to $x$ can be deduced from \cite{BuckdahnLiPengRainer_MFSDEandAssociatedPDE} under the assumption that $\overline{b}$ and $\varphi:\Rbb \to \Rbb$ are continuously differentiable with bounded Lipschitz derivatives. We here establish weak (Sobolev) differentiability of $X^x$ with respect to $x$ for the general drift $b$ given in~\eqref{introdMFSDE} by assuming in addition to \eqref{formDrift} that $\mu \mapsto b(t,y,\mu)$ is Lipschitz continuous uniformly in $t\in[0,T]$ and $y \in \Rbb$, i.e.~there exists a constant $C>0$ such that for all $t\in[0,T]$, $y \in \Rbb$ and $\mu, \nu \in \Pcal_1(\Rbb)$
\begin{align}\label{eq:uniformLipschitzThird}
	\vert b(t,y,\mu) - b(t,y,\nu)\vert \leq C \Kcal(\mu,\nu).
\end{align}
Further, also for the Sobolev derivative we provide a probabilistic representation in terms of local-time space integration.

The third main contribution of this paper is a Bismut-Elworthy-Li formula for first order derivatives of expectation functionals $\Ebb[\Phi(X_T^x)]$, ${\Phi: \Rbb \to \Rbb}$, of a strong solution $X^x$ of mean-field SDE \eqref{introdMFSDE}. 
Assuming the drift $b$ is in the form \eqref{formDrift} and fulfills the Lipschitz condition \eqref{eq:uniformLipschitzThird}, 
we first show Sobolev differentiability of these expectation functionals whenever $\Phi$ is continuously differentiable with bounded Lipschitz derivative. We then continue to develop a Bismut-Elworthy-Li type formula, that is we give a probabilistic representation for the first-order derivative of the form
	\begin{align}\label{BismutGeneral}
		\frac{\partial}{\partial x} \Ebb[\Phi(X_T^x)] = \Ebb\left[\Phi(X_T^x) \int_0^T \theta(t,X_t^x) dB_t\right],
	\end{align}
	where $\theta:[0,T] \times \Rbb \to \Rbb$ is some function independent of $\Phi$. This extends the result in \cite{Banos_Bismut}, where the author proves a Bismut-Elworthy-Li type formula for drift coefficients that are continuously differentiable in the space and law variable with bounded Lipschitz derivatives. We remark that compared to \cite{Banos_Bismut}, in addition to deal with irregular drift coefficients we are able to determine the so-called Malliavin weight $\int_0^T \theta(t,X_t^x) dB_t$ in terms of an It\^{o} integral and not in terms of an anticipative Skorohod integral.

Finally, we remark that in \cite{Bauer_RegularityOfMFSDE} we study (strong) solutions of mean-field SDEs and a corresponding Bismut-Elworthy-Li formula where the dependence of the drift $b$ on the solution law $\Pbb_{X_t^x}$ in \eqref{introdMFSDE} is of the special form
	\begin{align}\label{eq:specialMFSDEGeneral}
		dX_t^x = \overline{b}\left(t,X_t^x, \int_\Rbb \varphi(t,X_t^x,z) \Pbb_{X_t^x}(dz) \right) dt + dB_t, \quad X_0^x = x \in \Rbb, 
	\end{align}
	for some $\overline{b},\varphi:[0,T] \times \Rbb \times \Rbb\rightarrow \Rbb$. For this special class of mean-field SDEs, which includes the two popular drift families given in \eqref{specialMFSDE} and \eqref{eq:DriftVer}, we allow for irregularity of $\overline{b}$ and $\varphi$ that is not covered by our assumptions on $b$ in this paper. For example, for the indicator function $\varphi(t,x,z) = I_{z\le u}$ we are able to deal in \cite{Bauer_RegularityOfMFSDE} with the important case where the drift $\overline{b}\left(t,X_t^x,F_{X_t^x}(u)\right)$ depends on the distribution function $F_{X_t^x}(\cdot)$ of the solution. 
	
%

The remaining paper is organized as follows. In the second section we deal with existence and uniqueness of solutions of the mean-field SDE \eqref{introdMFSDE}. The third section investigates the afore-mentioned regularity properties of strong solutions. Finally, a proof of weak differentiability of expectation functionals $\Ebb[\Phi(X_T^x)]$ is given in the fourth section together with a Bismut-Elworthy-Li formula.\\[0.5cm]
\textbf{Notation:} Subsequently we list some of the most frequently used notations. For this, let $(\mathcal{X},d_{\mathcal{X}})$ and $(\mathcal{Y},d_{\mathcal{Y}})$ be two metric spaces.
\begin{itemize}
\item $\Ccal(\mathcal{X};\mathcal{Y})$ denotes the space of continuous functions $f:\mathcal{X} \to \mathcal{Y}$.
\item $\Ccal_0^{\infty}(U)$, $U\subseteq \Rbb$, denotes the space of smooth functions $f: U \to \Rbb$ with compact support.
\item For every $C>0$ we define the space $\Lip_C(\mathcal{X},\mathcal{Y})$ of functions $f:\mathcal{X}\to \mathcal{Y}$ such that
\begin{align*}
	d_{\mathcal{Y}}(f(x_1),f(x_2)) \leq C d_{\mathcal{X}}(x_1,x_2), \quad \forall x_1,x_2 \in \mathcal{X}
\end{align*}
as the space of Lipschitz functions with Lipschitz constant $C>0$. Furthermore, we define $\Lip(\mathcal{X},\mathcal{Y}) := \bigcup_{C>0} \Lip_C(\mathcal{X},\mathcal{Y})$ and denote by $\Lip_C(\mathcal{X}):= \Lip_C(\mathcal{X},\mathcal{X})$ and $\Lip(\mathcal{X}) := \Lip(\mathcal{X},\mathcal{X})$, respectively, the space of Lipschitz functions mapping from $\mathcal{X}$ to $\mathcal{X}$.
\item $\Ccal^{1,1}_{b,C}(\Rbb)$ denotes the space of continuously differentiable functions $f: \Rbb \to \Rbb$ with Lipschitz continuous and by $C>0$ bounded derivative $f'$, i.e.
	\begin{enumerate}[(a)]
		\item $\sup_{y\in \Rbb} |f'(y)| \leq C$, and
		\item $(y \mapsto f'(y)) \in \Lip_C(\Rbb)$.
	\end{enumerate}
	We define $\Ccal^{1,1}_b(\Rbb) := \bigcup_{C>0} \Ccal^{1,1}_{b,C}(\Rbb)$.
\item $\Ccal_b^{1,L}(\Rbb\times \Pcal_1(\Rbb))$ is the space of functions $f:[0,T] \times \Rbb \times \Pcal_1(\Rbb) \to \Rbb$ such that there exists a constant $C>0$ with
	\begin{enumerate}[(a)]
		\item $(y \mapsto f(t,y,\mu)) \in \Ccal_{b,C}^{1,1}(\Rbb)$ for all $t\in[0,T]$ and $\mu \in \Pcal_1(\Rbb)$, and
		\item $(\mu \mapsto f(t,y,\mu)) \in \Lip_C(\Pcal_1(\Rbb),\Rbb)$ for all $t\in[0,T]$ and $y\in \Rbb$.
	\end{enumerate}
\item Let $(\Omega, \Fcal, \Fbb, \Pbb)$ be a generic complete filtered probability space with filtration $\Fbb = (\Fcal_t)_{t\in[0,T]}$ and $B = (B_t)_{t\in[0,T]}$ be a Brownian motion defined on this probability space. Furthermore, we write $\Ebb[\cdot] := \Ebb_{\Pbb}[\cdot]$, if not mentioned differently.
\item $L^p(\Scal)$ denotes the Banach space of functions on the measurable space $(\Scal,\Gcal)$ integrable to some power $p$, $p\geq 1$.
\item $L^p(\Omega,\Fcal_t)$ denotes the space of $\Fcal_t$ measurable functions in $L^p(\Omega)$.
\item Let $f: \Rbb \to \Rbb$ be a (weakly) differentiable function. Then we denote by $\partial_y f(y):= \frac{\partial f}{\partial y} (y)$ its first (weak) derivative evaluated at $y \in \Rbb$.
\item We denote the Doléan-Dade exponential for a progressive process $Y$ with respect to the corresponding Brownian integral if well-defined for $t\in[0,T]$ by $$\Ecal\left( \int_0^t Y_u dB_u \right) := \exp \left\lbrace \int_0^t Y_u dB_u - \frac{1}{2} \int_0^t |Y_u|^2du \right\rbrace.$$
\item We define $B_t^x := x + B_t$, $t\in [0,T]$, for any Brownian motion $B$. 
\item For any normed space $\mathcal{X}$ we denote its corresponding norm by $\Vert \cdot \Vert_{\mathcal{X}}$; the Euclidean norm is denoted by $|\cdot |$.
\item We write $E_1(\theta) \lesssim E_2(\theta)$ for two mathematical expressions $E_1(\theta),E_2(\theta)$ depending on some parameter $\theta$, if there exists a constant $C>0$ not depending on $\theta$ such that $E_1(\theta) \leq C E_2(\theta)$.
\item We denote by $L^{X}$ the local time of the stochastic process $X$ and furthermore by $\int_s^t \int_{\Rbb} b(u,y) L^X(du, dy)$ for suitable $b$ the local-time space integral as introduced in \cite{Eisenbaum_LocalTimeIntegral} and extended in \cite{MeyerBrandisBanosDuedahlProske_ComputingDeltas}.
\item We denote the Wiener transform of some $Z \in L^2(\Omega,\Fcal_T)$ in $f \in L^2([0,T])$ by
\begin{align*}
	\Wcal(Z)(f) := \Ebb \left[ Z \Ecal\left(\int_0^T f(s) dB_s \right) \right].
\end{align*}
\end{itemize}

\section{Existence and Uniqueness of Solutions }\label{SectionStrongSolution}
The main objective of this section is to investigate existence and uniqueness of strong solutions of the one-dimensional mean-field SDE
\begin{align}\label{mainMcKeanVlasov}
	dX_t^x = b(t,X_t^x,\Pbb_{X_t^x}) dt + dB_t, \quad X_0^x = x \in \Rbb, \quad t\in [0,T],
\end{align}
with irregular drift coefficient $b: \Rbb^+ \times \Rbb \times \Pcal_1(\Rbb) \to \Rbb$. We first consider existence and uniqueness of weak solutions of \eqref{mainMcKeanVlasov} in \Cref{SSWeak}, which consecutively is employed together with results from \cite{MeyerBrandisBanosDuedahlProske_ComputingDeltas} to study strong solutions of \eqref{mainMcKeanVlasov} in \Cref{SSStrong}.

\subsection{Existence and Uniqueness of Weak Solutions} \label{SSWeak}

We recall the definition of weak solutions.

\begin{definition}
	A \emph{weak solution} of the mean-field SDE \eqref{mainMcKeanVlasov} is a six-tuple \linebreak${(\Omega, \Fcal, \Fbb, \Pbb, B, X^x)}$ such that
	\begin{enumerate}[(i)]
	\item $(\Omega, \Fcal, \Pbb)$ is a complete probability space and $\Fbb = \lbrace \Fcal_t \rbrace_{t\in [0,T]}$ is a filtration on $(\Omega, \Fcal, \Pbb)$ satisfying the usual conditions of right-continuity and completeness,
	\item $X^x = (X^x_t)_{t\in [0,T]}$ is a continuous, $\Fbb$-adapted, $\Rbb$-valued process; $B=(B_t)_{t\in [0,T]}$ is a one-dimensional $(\Fbb, \Pbb)$-Brownian motion,
	\item $X^x$ satisfies $\Pbb$-a.s.
	\begin{align*}
		dX^x_t = b(t,X^x_t,\Pbb_{X^x_t}) dt + d B_t, \quad X^x_0 = x \in \Rbb, \quad t\in[0,T],
	\end{align*}
	where for all $t\in[0,T]$, $\Pbb_{X^x_t} \in \Pcal_1(\Rbb)$ denotes the law of $X^x_t$ with respect to $\Pbb$.
	\end{enumerate}
\end{definition}
\begin{remark}
If there is no ambiguity about the stochastic basis $(\Omega, \Fcal, \Fbb, \Pbb, B)$ we also refer solely to the process $X^x$ as weak solution (or later on as strong solution) for notational convenience.
\end{remark}

In a first step we employ Girsanov's theorem in a well-known way to construct weak solutions of certain stochastic differential equations (hereafter SDE) associated to our mean-field SDE \eqref{mainMcKeanVlasov}. Assume the drift coefficient $b: [0,T] \times \Rbb \times \Pcal_1(\Rbb) \to \Rbb$ satisfies the linear growth condition \eqref{linearGrowth}. For a given $\mu \in \Ccal([0,T];\Pcal_1(\Rbb))$ we then define $b^{\mu}: \Rbb^+ \times \Rbb \to \Rbb$ by $b^{\mu}(t,y) := b(t,y,\mu_t)$ and consider the SDE
 \begin{align}\label{helpMcKeanVlasov}
	 	dX_t^{x} = b^{\mu}(t,X_t^{x})dt + dB_t, \quad X_0^{x} = x \in \Rbb, \quad t\in[0,T].
\end{align}
Let $\tilde{B}$ be a one-dimensional Brownian motion on a suitable filtered probability space $(\Omega,\Fcal,\Fbb,\Qbb)$. Define $X_t^x := \tilde{B}_t + x$. By \Cref{equivalentBrownianMeasurePbb}, the density $\frac{d\Pbb^{\mu}}{d\Qbb} = \Ecal\left( \int_0^T b^\mu(t,\tilde{B}_t^x) d\tilde{B}_t \right)$ gives rise to a well-defined equivalent probability measure $\Pbb^{\mu}$, and by Girsanov's theorem $B^{\mu}_t := X_t^x  - x - \int_0^t b^\mu(s,X_s^{x,\mu} )ds$, $t\in[0,T] $, defines an $(\Fbb,\Pbb^{\mu})$-Brownian motion. Hence, $(\Omega, \Fcal, \Fbb, \Pbb^{\mu}, B^{\mu}, X_t^x)$ is a weak solution of SDE \eqref{helpMcKeanVlasov}.

To show existence of weak solutions of the mean-field SDE \eqref{mainMcKeanVlasov} we proceed by employing the weak solutions of the auxiliary SDEs in \eqref{helpMcKeanVlasov} together with a fixed point argument. The upcoming theorem is a modified version of Theorem 3.2 in \cite{LiMin_WeakSolutions} for non-path-dependent coefficients, where we extend the assumptions on the drift from boundedness to linear growth. 

\begin{theorem}\label{existenceWeakSolutionMFSDE}
	Let the drift coefficient $b: [0,T] \times \Rbb \times \Pcal_1(\Rbb) \to \Rbb$ be a measurable function that satisfies conditions \eqref{linearGrowth} and \eqref{continuityThirdVariable}, i.e.~$b$ is of at most linear growth and continuous in the third variable. Then there exists a weak solution of the mean-field SDE \eqref{mainMcKeanVlasov}.
	Furthermore, ${\Pbb_{X_{\cdot}^x} \in \Ccal([0,T];\Pcal_1(\Rbb))}$ for any weak solution $ X^x$ of \eqref{mainMcKeanVlasov}.
\end{theorem}
\begin{proof}
We will state the proof just in the parts that differ from the proof in \cite{LiMin_WeakSolutions}. \\
For $\mu\in\Ccal([0,T];\Pcal_1(\Rbb))$ let $(\Omega, \Fcal, \Fbb, \Pbb^{\mu}, B^{\mu}, X^{x,\mu})$ be a weak solution of SDE \eqref{helpMcKeanVlasov}. We define the mapping $\psi: \Ccal([0,T]; \Pcal_1(\Rbb)) \to \Ccal([0,T]; \Pcal_1(\Rbb))$ by
	\begin{align*}
		\psi_s(\mu) := \Pbb_{X_s^{x,\mu}}^{\mu},
	\end{align*}	 
where $\Pbb_{X_s^{x,\mu}}^{\mu}$ denotes the law of $X_s^{x,\mu}$ under $\Pbb^{\mu}$, $s\in[0,T]$. Note that it can be shown equivalently to (ii) below that $\psi_s(\mu)$ is indeed continuous in $s\in[0,T]$. We need to show that $\psi$ has a fixed point, i.e. $\mu_s = \psi_s(\mu) = \Pbb_{X_s^{x,\mu}}^{\mu}$ for all $s\in[0,T]$. To this end we aim at applying Schauder's fixed point theorem (cf.~\cite{Schauder}) to $\psi:E \to E$, where 
\begin{align*}
	E:= \left\lbrace \mu \in \Ccal([0,T];\Pcal_1(\Rbb)): \Kcal(\mu_t, \delta_x) \leq C,\text{ } \Kcal(\mu_t,\mu_s) \leq C|t-s|^{\frac{1}{2}},\text{ } t,s\in[0,T] \right\rbrace,
\end{align*}
for some suitable constant $C>0$. Therefore we have to show that $E$ is a non-empty convex subset of $\Ccal([0,T];\Pcal_1(\Rbb))$, $\psi$ maps $E$ continuously into $E$ and $\psi(E)$ is compact.  Due to the proof of Theorem 3.2 in \cite{LiMin_WeakSolutions} it is left to show that for all $s,t\in[0,T]$ and $\mu \in E$,
\begin{enumerate}[(i)]
\item $\psi$ is continuous on $E$,
\item $\Kcal(\psi_t(\mu),\psi_s(\mu)) \lesssim \vert t-s \vert^{\frac{1}{2}},$
\item $\Ebb_{\Pbb^{\mu}}[\vert X_t^{\mu,x} \vert \mathbbm{1}_{\lbrace \vert X_t^{\mu,x} \vert \geq r \rbrace}] \xrightarrow[r\to \infty]{} 0$.
\end{enumerate}

\begin{enumerate}[(i)]
\item First note that $E$ endowed with $\sup_{t\in[0,T]} \Kcal(\cdot,\cdot)$, is a metric space. Let $\tilde{\varepsilon}>0$, $\mu \in E$ and $C_1>0$ be some constant. Moreover, let $C_{p,T}>0$ be a constant depending on $p$ and $T$ such that by Burkholder-Davis-Gundy's inequality $\Ebb \left[ |B_t|^{2p} \right]^{\frac{1}{2p}} \leq \frac{C_{p,T}}{2C_1}$ for all $t\in[0,T]$. Since $b$ is continuous in the third variable and $\cdot^2$ is a continuous function, we can find $\delta_1> 0$ such that for all $\nu \in E$ with $\sup_{t\in[0,T]} \Kcal(\mu_t,\nu_t) < \delta_1$,
	\begin{align}\label{boundContinuity1}
	\begin{split}
		\sup_{t\in[0,T], y\in\mathbb{R}} \vert b(t,y, \mu_t) - b(t,y, \nu_t) \vert &< \frac{\tilde{\varepsilon}}{2C_{p,T}T^{\frac{1}{2}}}, \\
		\sup_{t\in[0,T], y\in\mathbb{R}} \left| \vert b(t,y, \mu_t)\vert^2 - \vert b(t,y, \nu_t)\vert^2 \right\vert &< \frac{\tilde{\varepsilon}}{C_{p,T}T}.
	\end{split}
	\end{align}
	Furthermore, by the proof of \Cref{boundMeasureChange} we can find $\varepsilon>0$ such that
	\begin{align}\label{boundContinuity2}
		\sup_{\lambda \in E} \Ebb &\left[ \Ecal\left( - \int_0^T b(t,B_t^x,\lambda_t)dB_t \right)^{1+\varepsilon} \right]^{\frac{1}{1+\varepsilon}} \leq C_1.
	\end{align}
	Then, we get by the definition of $\psi$ that
\begin{align*}
	\Kcal&(\psi_t(\mu), \psi_t(\nu)) = \sup_{h \in \text{Lip}_1} \left\lbrace \left\vert \int_{\Rbb} h(y) \psi_t(\mu)(dy) - \int_{\Rbb} h(y) \psi_t(\nu)(dy) \right\vert \right\rbrace \\
	&= \sup_{h \in \text{Lip}_1} \left\lbrace \left\vert \int_{\Rbb} (h(y)-h(x)) \left( \Pbb^{\mu}_{X_t^{x,\mu}}-\Pbb^{\nu}_{X_t^{x,\nu}} \right)(dy) \right\vert \right\rbrace \\
	&= \sup_{h \in \text{Lip}_1} \left\lbrace \left\vert \Ebb_{\Qbb^{\mu}}\left[ \left(h(X_t^{x,\mu})-h(x)\right) \Ecal_t(\mu)\right] - \Ebb_{\Qbb^{\nu}}\left[ \left(h(X_t^{x,\nu})-h(x)\right) \Ecal_t(\nu)\right] \right\vert \right\rbrace \\
	&\leq  \Ebb\left[ \left\vert \Ecal_t(\mu) - \Ecal_t(\nu) \right\vert |B_t|\right],
\end{align*}
where $\frac{d\Qbb^{\mu}}{d\Pbb^{\mu}} = \Ecal\left( - \int_0^t b(s,X_s^{x,\mu},\mu_s)dB_s^{\mu} \right)$ defines an equivalent probability measure $\Qbb^{\mu}$ by \Cref{equivalentBrownianMeasurePbb}. Here we have used $\Ecal_t(\mu) := \Ecal\left(\int_0^t b(s,B_s^x,\mu_s)dB_s \right)$
and the fact that $X^{x,\mu}$ is a Brownian motion under $\Qbb^{\mu}$ starting in $x$ for all $\mu \in \Ccal([0,T];\Pcal(\Rbb))$. We get by the inequality 
\begin{align}\label{exponentialInequality}
	|e^y-e^z| \leq |y-z|(e^y+e^z), \quad y,z \in \Rbb,
\end{align}
	Hölder's inequality with $p:= \frac{1+\varepsilon}{\varepsilon}$, $\varepsilon >0$ sufficiently small with regard to \eqref{boundContinuity2}, and Minkowski's inequality that
\begin{align}\label{eq:estimateKantorovich}
\begin{split}
	\Kcal&(\psi_t(\mu), \psi_t(\nu)) \leq \Ebb \left[ |B_t| \left(\Ecal\left(\int_0^t b(s,B_s^x,\mu_s)dB_s \right)+\Ecal\left(\int_0^t b(s,B_s^x,\nu_s)dB_s \right)\right)\right.\\ 
	&\quad \times \left.\left\vert \int_0^t b(s,B_s^x, \mu_s) - b(s,B_s^x, \nu_s) dB_s  - \frac{1}{2} \int_0^t \vert b(s,B_s^x, \mu_s)\vert^2 - \vert b(s,B_s^x, \nu_s)\vert^2 ds \right\vert \right] \\
	&\leq \left( \Ebb \left[ \Ecal\left(\int_0^t b(s,B_s^x,\mu_s)dB_s \right)^{1+\varepsilon} \right]^{\frac{1}{1+\varepsilon}} + \Ebb\left[ \Ecal\left(\int_0^t b(s,B_s^x,\nu_s)dB_s \right)^{1+\varepsilon} \right]^{\frac{1}{1+\varepsilon}} \right) \\
	&\quad \times \left( \Ebb \left[ \left(\int_0^t \vert b(s,B_s^x, \mu_s) - b(s,B_s^x, \nu_s) \vert dB_s\right)^{2p} \right]^{\frac{1}{2p}} \right. \\
	&\quad + \left. \frac{1}{2} \Ebb\left[ \left( \int_0^t \left\vert \vert b(s,B_s^x, \mu_s)\vert^2 - \vert b(s,B_s^x, \nu_s)\vert^2 \right\vert ds \right)^{2p} \right]^{\frac{1}{2p}} \right) \Ebb \left[ |B_t|^{2p} \right]^{\frac{1}{2p}}.
\end{split}
\end{align}
Consequently, we get by Burkholder-Davis-Gundy's inequality and the bounds in \eqref{boundContinuity1} and \eqref{boundContinuity2} that
\begin{align*}
	\sup_{t\in[0,T]} \Kcal(\psi_t(\mu), \psi_t(\nu)) &\leq  C_{p,T} \left( \Ebb\left[ \left( \int_0^T \vert b(s,B_s^x, \mu_s) - b(s,B_s^x, \nu_s) \vert^2 ds\right)^p \right]^{\frac{1}{2p}} \right. \\
	& \quad + \left. \frac{1}{2} \Ebb \left[ \left( \int_0^T \left\vert \vert b(s,B_s^x, \mu_s)\vert^2 - \vert b(s,B_s^x, \nu_s)\vert^2 \right\vert ds \right)^{2p} \right]^{\frac{1}{2p}} \right)\\
	&< T^{\frac{1}{2}} \frac{\tilde{\varepsilon}}{2T^{\frac{1}{2}}} + \frac{T}{2} \frac{\tilde{\varepsilon}}{T} = \tilde{\varepsilon}.
\end{align*}
Hence, $\psi$ is continuous on $E$.
\item Define $p := \frac{1+\varepsilon}{\varepsilon}$, $\varepsilon > 0$ sufficiently small with regard to \eqref{boundContinuity2}, and let $\mu \in E$ and $s,t \in[0,T]$ be arbitrary. Then, equivalently to \eqref{eq:estimateKantorovich}
\begin{align*}
	\Kcal(\psi_t(\mu),\psi_s(\mu)) &\leq \Ebb\left[ \left\vert \Ecal_t(\mu) - \Ecal_s(\mu) \right\vert |B_t| \right] \\
	&\lesssim \Ebb\left[ \left\vert \int_s^t b(r,B_r^x, \mu_r) dB_r - \frac{1}{2} \int_s^t \vert b(r,B_r^x, \mu_r) \vert^2 dr \right\vert^{2p} \right]^{\frac{1}{2p}}.
\end{align*}
Furthermore, by applying Burkholder-Davis-Gundy's inequality, we get
\begin{align*}
	\Kcal&(\psi_t(\mu),\psi_s(\mu)) \lesssim \Ebb\left[ \left( \int_s^t \left\vert b(r,B_r^x, \mu_r)\right\vert^2 dr\right)^p \right]^{\frac{1}{2p}} + \Ebb\left[ \left( \int_s^t \left\vert b(r,B_r^x, \mu_r)\right\vert^2 dr\right)^{2p} \right]^{\frac{1}{2p}} \\
	&\leq \Ebb\left[|t-s|^p  \sup_{r\in[0,T]} \left\vert b(r,B_r^x, \mu_r)\right\vert^{2p} \right]^{\frac{1}{2p}} + \Ebb\left[|t-s|^{2p} \sup_{r\in[0,T]} \left\vert b(r,B_r^x, \mu_r)\right\vert^{4p} \right]^{\frac{1}{2p}}.
\end{align*}
Finally by \Cref{pNorm}, we get that 
\begin{align*}
	\Kcal(\psi_t(\mu),\psi_s(\mu)) &\leq C_2 \left( |t-s|^{\frac{1}{2}}+ |t-s| \right) \lesssim |t-s|^{\frac{1}{2}},
\end{align*}
for some constant $C_2>0$, which is independent of $\mu \in E$.
\item The claim holds by \Cref{pNorm} and dominated convergence for $r\to \infty$.
\end{enumerate}
\end{proof}

Next, we study uniqueness of weak solutions. We recall the definition of weak uniqueness, also called uniqueness in law.

\begin{definition}
	We say a weak solution $(\Omega^1, \Fcal^1, \Fbb^1, \Pbb^1, B^1, X^1)$ of \eqref{mainMcKeanVlasov} is \emph{weakly unique} or \emph{unique in law}, if for any other weak solution $(\Omega^2, \Fcal^2, \Fbb^2, \Pbb^2, B^2, X^2)$ of \eqref{mainMcKeanVlasov} it holds that
	\begin{align*}
		\Pbb^1_{X^1} = \Pbb^2_{X^2},
	\end{align*}
	whenever $X^1_0 = X^2_0$.
\end{definition}

In order to establish weak uniqueness we have to make further assumptions on the drift coefficient.

\begin{definition}\label{modulusOfContinuity}
	Let $b: [0,T] \times \Rbb \times \Pcal_1(\Rbb) \to \Rbb$ be a measurable function. We say $b$ admits $\theta$ as a modulus of continuity in the third variable, if there exists a continuous function $\theta:\Rbb^+ \to \Rbb^+$, with $\theta(y)>0$ for all $y\in \Rbb^+$, $\int_0^z \frac{dy}{\theta(y)} = \infty$ for all $z \in \Rbb^+$, and for all $t\in[0,T]$, $y \in \Rbb$ and $\mu,\nu \in \Pcal_1(\Rbb)$,
	\begin{align}\label{modulusOfContinuityEq}
		|b(t,y,\mu)-b(t,y,\nu)|^2 \leq \theta(\Kcal(\mu,\nu)^2).
	\end{align}
\end{definition}

\begin{remark}
	Note that this definition is a special version of the general definition of modulus of continuity. In general one requires $\theta$ to satisfy $\lim_{x\to 0} \theta(x) = 0$ and for all $t\in[0,T]$, $y \in \Rbb$ and $\mu,\nu \in \Pcal_1(\Rbb)$,
	\begin{align*}
		|b(t,y,\mu)-b(t,y,\nu)| \leq \theta(\Kcal(\mu,\nu)).
	\end{align*}
	 It is readily verified that if $b$ admits $\theta$ as a modulus of continuity according to \Cref{modulusOfContinuity} it also admits one in the sense of the general definition.
\end{remark}

\begin{theorem}\label{weakUniqueness}
	Let the drift coefficient $b: [0,T] \times \Rbb \times \Pcal_1(\Rbb) \to \Rbb$ satisfy conditions \eqref{linearGrowth} and \eqref{modulusOfContinuityEq}, i.e.~$b$ is of at most linear growth and admits a modulus of continuity in the third variable. Let $(\Omega^i, \Fcal^i, \Fbb^i, \Pbb^i,B^i,X^i)$, $i=1,2$, be two weak solutions of \eqref{mainMcKeanVlasov}. Then 
	\begin{align*}
		\Pbb_{(X^1,B^1)}^1 = \Pbb_{(X^2,B^2)}^2.
	\end{align*}
	In particular the solutions are unique in law.
\end{theorem}
\begin{proof}
	For the sake of readability we just consider the case $x=0$. The general case follows in the same way. From \Cref{equivalentBrownianMeasurePbb} and Girsanov's theorem, we know that there exist measures $\Qbb^1$ and $\Qbb^2$ under which $X^1$ and $X^2$ are Brownian motions, respectively. Similarly to the idea in the proof of Theorem 4.2 in \cite{LiMin_WeakSolutions}, we define by \Cref{equivalentBrownianMeasurePbb} an equivalent probability measure $\tilde{\Qbb}^2$ by
	\begin{align*}
		\frac{d\tilde{\Qbb}^2}{d\Pbb^2} := \Ecal \left( -\int_0^T \left(b(s,X_s^2,\Pbb_{X_s^2}^2) - b(s,X_s^2,\Pbb_{X_s^1}^1) \right) dB_s^2 \right),
	\end{align*}
	and the $\tilde{\Qbb}^2$-Brownian motion
	\begin{align*}
		\tilde{B}_t^2 := B_t^2 + \int_0^t b(s,X_s^2,\Pbb_{X_s^2}^2) - b(s,X_s^2,\Pbb_{X_s^1}^1) ds, \quad t\in[0,T].
	\end{align*}
	Since
	\begin{align*}
		B_t^1 &= X_t^1 - \int_0^t b(s,X_s^1,\Pbb_{X^1_s}^1) ds \quad \text{ and} \quad \tilde{B}_t^2 = X_t^2 - \int_0^t b(s,X_s^2,\Pbb_{X^1_s}^1) ds,
	\end{align*}
	we can find a measurable function $\Phi:[0,T] \times \Ccal([0,T];\Rbb) \to \Rbb$ such that
	\begin{align*}
		B^1_t = \Phi_t(X^1) \quad \text{ and } \quad \tilde{B}^2_t = \Phi_t(X^2).
	\end{align*}
	Recall that $X^i$ is a $\Qbb^i$-Brownian motion, $i=1,2$. Consequently we have for every bounded measurable functional $F: \Ccal([0,T];\Rbb) \times \Ccal([0,T];\Rbb) \to \Rbb$
	\begin{align*}
		\Ebb_{\Pbb^1} [F(B^1,X^1)] &= \Ebb_{\Qbb^1}\left[\Ecal\left(\int_0^T b(t,X_t^1,\Pbb_{X_t^1}^1) dX_t^1\right) F(\Phi(X^1),X^1)\right] \\
		&= \Ebb_{\Qbb^2}\left[\Ecal\left(\int_0^T b(t,X_t^2,\Pbb_{X_t^1}^1) dX_t^2\right) F(\Phi(X^2),X^2)\right] \\
		&= \Ebb_{\tilde{\Qbb}^2}[F(\tilde{B}^2,X^2)].
	\end{align*}
	Hence,
	\begin{align}\label{measuresEqual}
		\Pbb^1_{(X^1,B^1)} = \tilde{\Qbb}^2_{(X^2,\tilde{B}^2)}.
	\end{align}
	It is left to show that $\sup_{t\in[0,T]} \Kcal(\tilde{\Qbb}^2_{X_t^2},\Pbb_{X_t^2}^2) = 0$, from which we conclude together with \eqref{measuresEqual} that $\sup_{t\in[0,T]} \Kcal(\Pbb^1_{X_t^1},\Pbb_{X_t^2}^2) = 0$ and hence $\frac{d\tilde{\Qbb}^2}{d\Pbb^2}=1$.  Consequently, $\Pbb^1_{(X^1,B^1)} = \Pbb^2_{(X^2,B^2)}$. \par
	Using Hölder's inequality, we get for $p:=\frac{1+\varepsilon}{\varepsilon}$, $\varepsilon>0$ sufficiently small with regard to \Cref{boundsSolution},
	\begin{align*}
		\Kcal&(\tilde{\Qbb}^2_{X_t^2},\Pbb_{X_t^2}^2) = \sup_{h\in\Lip_1} \left| \Ebb_{\tilde{\Qbb}^2} \left[h(X_t^2)-h(0) \right] - \Ebb_{\Pbb^2}\left[ h(X_t^2)-h(0)\right] \right| \\
		&\leq \sup_{h\in\Lip_1} \Ebb_{\Pbb^2} \left[\left|\Ecal \left( -\int_0^t \left(b(s,X_s^2,\Pbb_{X_s^2}^2) - b(s,X_s^2,\Pbb_{X_s^1}^1) \right) dB_s^2 \right)-1\right| \left| h\left(X_t^2\right)-h(0)\right|\right] \\
		&\leq \Ebb_{\Pbb^2} \left[\left|\Ecal \left( -\int_0^t \left(b(s,X_s^2,\Pbb_{X_s^2}^2) - b(s,X_s^2,\Pbb_{X_s^1}^1) \right) dB_s^2 \right)-1\right|^{\frac{2(1+\varepsilon)}{2+\varepsilon}}\right]^{\frac{2+\varepsilon}{2(1+\varepsilon)}}\\
		&\quad \times \Ebb\left[ \Ecal\left(\int_0^t b(s,B_s,\Pbb^2_{X_s^2}) dB_s \right)^{1+\varepsilon}\right]^{\frac{\varepsilon}{2(1+\varepsilon)^2}} \Ebb\left[ |B_t|^{2p^2}\right]^{\frac{1}{2p^2}} \\
		&\lesssim \Ebb_{\Pbb^2} \left[\left|\Ecal \left( -\int_0^t \left(b(s,X_s^2,\Pbb_{X_s^2}^2) - b(s,X_s^2,\Pbb_{X_s^1}^1) \right) dB_s^2 \right)-1\right|^{\frac{2(1+\varepsilon)}{2+\varepsilon}}\right]^{\frac{2+\varepsilon}{2(1+\varepsilon)}}.
	\end{align*}
	Using that $b$ admits a modulus of continuity in the third variable, we get by inequality \eqref{exponentialInequality}, \Cref{boundsSolution}, and Burkholder-Davis-Gundy's inequality that
	\begin{align*}
		\Kcal(\tilde{\Qbb}^2_{X_t^2},\Pbb_{X_t^2}^2) &\lesssim \Ebb_{\Pbb^2} \left[\left|\exp\left\lbrace -\int_0^t \left(b(s,X_s^2,\Pbb_{X_s^2}^2) - b(s,X_s^2,\Pbb_{X_s^1}^1) \right) dB_s^2 \right.\right.\right.\\
		&\left.\left.\left. \quad -\frac{1}{2} \int_0^t \left(b(s,X_s^2,\Pbb_{X_s^2}^2) - b(s,X_s^2,\Pbb_{X_s^1}^1) \right)^2 ds\right\rbrace-\exp\lbrace 0 \rbrace \right|^{\frac{2(1+\varepsilon)}{2+\varepsilon}}\right]^{\frac{2+\varepsilon}{2(1+\varepsilon)}}\\
		&\lesssim \Ebb_{\Pbb^2} \left[\left|\int_0^t \left(b(s,X_s^2,\Pbb_{X_s^2}^2) - b(s,X_s^2,\Pbb_{X_s^1}^1) \right) dB_s^2 \right.\right.\\
		&\left.\left. \quad + \frac{1}{2} \int_0^t \left(b(s,X_s^2,\Pbb_{X_s^2}^2) - b(s,X_s^2,\Pbb_{X_s^1}^1) \right)^2 ds \right|^{2p}\right]^{\frac{1}{2p}}\\
		&\lesssim \Ebb_{\Pbb^2} \left[\left|\int_0^t \left(b(s,X_s^2,\Pbb_{X_s^2}^2) - b(s,X_s^2,\Pbb_{X_s^1}^1) \right)^2 ds \right|^{p}\right]^{\frac{1}{2p}} \\
		&\quad + \Ebb_{\Pbb^2} \left[\left|\int_0^t \left(b(s,X_s^2,\Pbb_{X_s^2}^2) - b(s,X_s^2,\Pbb_{X_s^1}^1) \right)^2 ds \right|^{2p}\right]^{\frac{1}{2p}}\\
		&\leq \left( \int_0^t \theta\left( \Kcal(\tilde{\Qbb}^2_{X^2_s},\Pbb_{X_s^2}^2)^2\right) ds \right)^{\frac{1}{2}} + \int_0^t \theta\left( \Kcal(\tilde{\Qbb}^2_{X^2_s},\Pbb_{X_s^2}^2)^2\right) ds.
	\end{align*}
	Assume $\int_0^t \theta\left( \Kcal(\tilde{\Qbb}^2_{X^2_s},\Pbb_{X_s^2}^2)^2\right) ds \geq 1$. Then, 
	\begin{align*}
		\Kcal(\tilde{\Qbb}^2_{X_t^2},\Pbb_{X_t^2}^2)^2 \lesssim \int_0^t \tilde{\theta}\left( \Kcal(\tilde{\Qbb}^2_{X^2_s},\Pbb_{X_s^2}^2)^2\right) ds,
	\end{align*}
	where for all $z\in \Rbb^+$, $\tilde{\theta}:= \theta^2$ satisfies the assumption $\int_0^z \frac{1}{\tilde{\theta}(y)} dy = \infty$.\\
	In the case $0\leq \int_0^t \theta\left( \Kcal(\tilde{\Qbb}^2_{X^2_s},\Pbb_{X_s^2}^2)^2\right) ds < 1$, we get
	\begin{align*}
		\Kcal(\tilde{\Qbb}^2_{X_t^2},\Pbb_{X_t^2}^2)^2 \lesssim \int_0^t \theta\left( \Kcal(\tilde{\Qbb}^2_{X^2_s},\Pbb_{X_s^2}^2)^2\right) ds.
	\end{align*}
	We know that $t \mapsto \Kcal(\tilde{\Qbb}_{X_t^2}^2,\Pbb_{X_t^2}^2)$ is continuous by the proof of \cite{LiMin_WeakSolutions}[Theorem 4.2] and of \Cref{existenceWeakSolutionMFSDE}. Hence, by Bihari's inequality (cf. \cite{Mao_AdaptedSolutionsOfBSDEwithNonLipschitzCoefficients}[Lemma 3.6]) $\Kcal(\tilde{\Qbb}^2_{X^2_t},\Pbb_{X_t^2}^2) = 0$ for all $t\in[0,T]$, which completes the proof.	
\end{proof}

\subsection{Existence and Uniqueness of Strong Solutions}\label{SSStrong} 

We recall the definition of a strong solution.

\begin{definition}\label{DefStrS}
	A \emph{strong solution} of the mean-field SDE \eqref{mainMcKeanVlasov} is a weak solution \linebreak $(\Omega, \Fcal, \Fcal^B, \Pbb, B, X^x)$ where $\Fcal^B$ is the filtration generated by the Brownian motion $B$ and augmented with the $\Pbb$-null sets.
\end{definition}

\begin{remark}\label{RemDefStr}
	Note that according to \Cref{DefStrS}, we say that \eqref{mainMcKeanVlasov} has a strong solution as soon as there exists some stochastic basis $(\Omega,\Fcal,\Pbb,B)$ with a brownian-adapted solution $X^x$, while usually in the literature the definition of a strong solution requires the (a priori stronger) existence of a brownian-adapted solution of \eqref{mainMcKeanVlasov} on any given stochastic basis.  However, in our setting these two definitions are equivalent. Indeed, a given strong solution $(\Omega, \Fcal, \Fcal^B, \Pbb, B, X^x)$ of the mean-field SDE \eqref{mainMcKeanVlasov} can be considered a strong solution of the associated SDE
\begin{align}\label{rewrittenSDE}
			dX_t^x = b^{\Pbb_{X}}(t,X_t^x) dt + dB_t, \quad X_0^x = x, \quad t\in [0,T],
\end{align}
where we define the drift coefficient $b^{\Pbb_{X}}: [0,T] \times \Rbb \to \Rbb$ by
\begin{align*}
		b^{\Pbb_{X}}(t,y) := b(t,y,\Pbb_{X_t^x}).
\end{align*}
For strong solutions of SDEs it is then well-known that  there exists a family of functionals $(F_t)_{t\in[0,T]}$ with $X_t^x = F_t(B)$ (see e.g. \cite{MeyerBrandisProske_OnTheExistenceOfStrongSolutionsOfLevyNoiseSDEs} for an explicit form of $F_t$), such that for any other stochastic basis $(\hat{\Omega}, \hat{\Fcal}, \hat{\Qbb}, \hat{B})$ the process $\hat{X}_t^x := F_t(\hat{B})$ is a $\Fcal^{\hat{B}}$-adapted solution of SDE \eqref{rewrittenSDE}. Further, from the functional form of the solutions we obviously get $\Pbb_{X}=\Pbb_{\hat{X}}$, and thus $b^{\Pbb_{X}}(t,y) = b^{\Pbb_{\hat{X}}}(t,y):=b(t,y,\Pbb_{\hat{X}_t^x})$, such that $\hat{X}^x$ fulfills
\begin{align*}
			d\hat{X}_t^x = b^{\Pbb_{\hat{X}}}(t,\hat{X}_t^x) dt + d\hat{B}_t, \quad \hat{X}_0^x = x, \quad t\in [0,T],
\end{align*}
i.e.~$(\hat{\Omega}, \hat{\Fcal}, \hat{\Qbb}, \hat{B}, \hat{X}^x)$ is a strong solution of the mean-field SDE \eqref{mainMcKeanVlasov}. Hence, the two definitions of strong solutions are equivalent. 
\end{remark}

In addition to weak uniqueness, a second type of uniqueness usually considered in the context of strong solutions is path-wise uniqueness:

\begin{definition}
	We say a weak solution $(\Omega, \Fcal, \Fbb, \Pbb, B^1, X^1)$ of \eqref{mainMcKeanVlasov} is \emph{path-wisely unique}, if for any other weak solution $(\Omega, \Fcal, \Fbb, \Pbb, B^2, X^2)$ on the same stochastic basis,
	\begin{align*}
		\Pbb\left( \forall t\geq 0: X^1_t = X^2_t \right) = 1.
	\end{align*}
\end{definition}

\begin{remark}\label{remarkUniqueness}
Note that in our setting weak uniqueness and path-wise uniqueness of strong solutions of the mean-field SDE \eqref{mainMcKeanVlasov} are equivalent. Indeed, any weakly unique strong solution of \eqref{mainMcKeanVlasov} is a weakly unique strong solution of the same associated SDE \eqref{rewrittenSDE}, i.e.~the drift coefficient in \eqref{rewrittenSDE} does not vary with the solution since the law of the solution is unique. Due to \cite{Cherny_UniquenessInLaw}[Theorem 3.2], a weakly unique strong solution of an SDE is always path-wisely unique, and thus a weakly unique strong solution of \eqref{mainMcKeanVlasov} is path-wisely unique. Vice versa, by the considerations in Remark~\ref{RemDefStr}, any path-wisely unique strong solution $(\Omega, \Fcal, \Pbb, B, X^x)$ of  \eqref{mainMcKeanVlasov} can be represented by $X_t^x = F_t(B)$ for some unique family of functionals $(F_t)_{t\in[0,T]}$ that does not vary with the stochastic basis.  Consequently, the strong solution is weakly unique. Thus, in the following we will just speak of a \emph{unique strong solution} of \eqref{mainMcKeanVlasov}.
\end{remark}

In order to establish existence of strong solutions we require in addition to the assumptions in \Cref{existenceWeakSolutionMFSDE} that the drift coefficient exhibits the particular linear growth given by the decomposable form \eqref{formDrift}, that is, the irregular behavior of the drift stays in a bounded spectrum.

\begin{theorem}\label{strongSolution}
	Suppose the drift coefficient $b$ is in the decomposable form \eqref{formDrift} and additionally continuous in the third variable, i.e.~fulfills \eqref{continuityThirdVariable}. Then there exists a strong solution of the mean-field SDE \eqref{mainMcKeanVlasov}. More precisely, any weak solution $(X_t^x)_{t\in [0,T]}$ of \eqref{mainMcKeanVlasov} is a strong solution, and in addition $X_t^x$ is Malliavin differentiable for every $t\in [0,T]$. \\
	If moreover $b$ satisfies \eqref{modulusOfContinuityEq}, i.e.~$b$ admits a modulus of continuity in the third variable, the solution is unique.
\end{theorem}
\begin{proof}
	Let $(\Omega, \Fcal, \Fbb, \Pbb, B, X^x)$ be a weak solution of the mean-field SDE \eqref{mainMcKeanVlasov}, which exists by \Cref{existenceWeakSolutionMFSDE}. Then $X^x$ is a weak solution of SDE~\eqref{rewrittenSDE}. \par 
Now we note that under the assumptions specified in \Cref{strongSolution} the drift $b^{\Pbb_{X}}(t,y)$ in \eqref{rewrittenSDE} satisfies the conditions required in \cite{MeyerBrandisBanosDuedahlProske_ComputingDeltas}[Theorem 3.1], from which it follows that there exists a unique strong solution of SDE \eqref{rewrittenSDE} that is Malliavin differentiable. Since by \Cref{equivalentBrownianMeasurePbb} the Radon-Nikodym density $\Ecal\left( - \int_0^t b^{\Pbb_{X}}(s,X_t^{x})dB_s \right)$ defines an equivalent probability measure under which $X^{x}$ is Brownian motion, it readily follows from the proof of  \cite{MeyerBrandisBanosDuedahlProske_ComputingDeltas}[Theorem 3.1] that $X^x$ must be this strong solution, and thus $X^x$ is a Malliavin differentiable strong solution of the mean-field SDE \eqref{mainMcKeanVlasov}. If further $b$ admits a modulus of continuity in the third variable, then by \Cref{weakUniqueness}, $X^x$ is a weakly (and by \Cref{remarkUniqueness} also path-wisely) unique strong solution of \eqref{mainMcKeanVlasov}.
\end{proof}


\section{Regularity properties}

We start this section by giving a probabilistic representation of the Malliavin derivative of a strong solution to the mean-field SDE \eqref{mainMcKeanVlasov}. If $b$ is Lipschitz continuous in the second variable, it is well-known that the Malliavin derivative is given by 
$ D_sX_t^x = \exp \left\lbrace \int_s^t \partial_2 b(u,X_u^x,\Pbb_{X_u^x}) du \right\rbrace$. For irregular drift $b$ we obtain the following generalized representation without the derivative of $b$ which is an immediate consequence of the proof of \Cref{strongSolution} and \cite{MeyerBrandisBanosDuedahlProske_ComputingDeltas}[Proposition 3.2]:

\begin{proposition}\label{derivativeMalliavin}
	Suppose the drift coefficient $b$ satisfies the assumptions of \Cref{strongSolution}. Then for $0\leq s \leq t \leq T$, the Malliavin derivative $D_sX_t^x$ of a strong solution $X^x$ to the mean-field SDE \eqref{mainMcKeanVlasov} has the following representation:
	\begin{align*}
		D_sX_t^x = \exp \left\lbrace - \int_s^t \int_{\Rbb} b(u,y,\Pbb_{X_u^x}) L^{X^x}(du,dy) \right\rbrace
	\end{align*}
	Here $L^{X^x}(du,dy)$ denotes integration with respect to local time of $X^x$ in time and space, see \cite{MeyerBrandisBanosDuedahlProske_ComputingDeltas} and \cite{Eisenbaum_LocalTimeIntegral} for more details.\\
\end{proposition}

In the remaining section we analyze the regularity of a strong solution $X^x$ of \eqref{mainMcKeanVlasov} in its initial condition $x$. More precisely, the main result in Theorem~\ref{solutionInSobolev} shows the existence of a weak (Sobolev) derivative $\partial_x X_t^x $ for irregular drift coefficients for every $t\in[0,T]$, which also is referred to as the \textit{first variation process}. We first show Lipschitz continuity of $X_t^x$ in $x$ for smooth coefficients $b$ in Porposition~\ref{differentiable}, which consecutively is employed together with an approximation argument and a compactness criterion to extend weak differentiability to more general coefficients $b$ in \Cref{solutionInSobolev}. Further, we give a probabilistic representation of the first variation process and establish a connection to the Malliavin derivative that will be employed to derive the Bismut-Elworthy-Li formula in Section~\ref{SBEL}. 
\begin{proposition}\label{differentiable}
	Let $b \in C_b^{1,L}(\Rbb\times \Pcal_1(\Rbb))$ and $X^x$ be the unique strong solution of mean-field SDE \eqref{mainMcKeanVlasov}. Then, for all $t\in [0,T]$ the map $x \mapsto X_t^x$ is a.s. Lipschitz continuous and consequently weakly and almost everywhere differentiable. Moreover, the first variation process $\partial_x X_t^x$, $t\in[0,T]$, has the representation
	\begin{align}\label{representationFV}
	\begin{split}
		\partial_x X_t^x &= \exp \left\lbrace \int_0^t \partial_2 b(s,X_s^x, \Pbb_{X_s^x}) ds \right\rbrace \\
		&\quad+ \int_0^t \exp \left\lbrace \int_u^t \partial_2 b(s,X_s^x, \Pbb_{X_s^x}) ds \right\rbrace \partial_x b(u,y,\Pbb_{X_u^x})\vert_{y= X_u^x} du.
	\end{split}
	\end{align}
\end{proposition}

\begin{remark}
	Note that compared to \cite{Banos_Bismut} we consider the more general case of mean-field SDEs of type \eqref{mainMcKeanVlasov} and therefore need to deal with differentiability of functions over the metric space $\Pcal_1(\Rbb)$ as in \cite{BuckdahnLiPengRainer_MFSDEandAssociatedPDE}, \cite{LasryLions_MeanFieldGames} and \cite{Cardaliaguet_NotesOnMeanFieldGames}. We avoid using the notion of differentiation with respect to a measure by considering the real function $x \mapsto b(t,y,\Pbb_{X_t^x})$, for which differentiation is understood in the Sobolev sense.
\end{remark}

\begin{proof}[Proof of \Cref{differentiable}]
		In order to prove Lipschitz continuity we have to show that there exists a constant $C>0$ such that for almost every $\omega \in \Omega$ and for all $t\in [0,T]$ the map $(x\mapsto X_t^x) \in \Lip_C(\Rbb)$. For notational reasons we hide $\omega$ in our computations and obtain using $b\in C_b^{1,L}(\Rbb\times\Pcal_1(\Rbb))$ that
	\begin{align}\label{solutionLipschitz}
	\begin{split}
		|X_t^x - X_t^y| &= \left| x-y + \int_0^t b(s,X_s^x,\Pbb_{X_s^x}) - b(s,X_s^y, \Pbb_{X_s^y}) ds \right| \\
		&\lesssim |x-y| + \int_0^t |X_s^x - X_s^y| + \Kcal(\Pbb_{X_s^x},\Pbb_{X_s^y}) ds.
		\end{split}
	\end{align}
	Hence, we immediately get that
	\begin{align*}
		\Kcal(\Pbb_{X_t^x},\Pbb_{X_t^y}) \leq \Ebb[|X_t^x - X_t^y|] \lesssim |x-y| + \int_0^t \Kcal(\Pbb_{X_s^x},\Pbb_{X_s^y}) ds,
	\end{align*}
	and therefore by Grönwall's inequality that
	\begin{align}\label{measureLipschitz}
		\Kcal(\Pbb_{X_s^x},\Pbb_{X_s^y}) \lesssim |x-y|.
	\end{align}
	Consequently, \eqref{solutionLipschitz} simplifies to 
	\begin{align}\label{contX}
		|X_t^x - X_t^y| \lesssim |x-y| + \int_0^t |X_s^x - X_s^y| ds,
	\end{align}
	and again by Grönwall's inequality we get that $(x\mapsto X_t^x) \in \Lip_C(\Rbb)$. Note that due to \eqref{measureLipschitz} and the assumptions on $b$ also $x \mapsto b(t,y,\Pbb_{X_t^x})$ is weakly differentiable for every $t \in [0,T]$ and $y \in \Rbb$.\\
	Regarding representation \eqref{representationFV}, note first that by taking the derivative with respect to $x$ in \eqref{mainMcKeanVlasov}, $\partial_x X_t^x$ has the representation
	\begin{align}\label{implRepresentationFV}
		\partial_x X_t^x = 1+ \int_0^t \partial_2 b(s,X_s^x, \Pbb_{X_s^x}) \partial_x X_s^x + \partial_x b(s,y,\Pbb_{X_s^x})\vert_{y= X_s^x} ds.
	\end{align}
	It is readily seen that \eqref{representationFV} solves this ODE $\omega$-wise and therefore is a representation of the first variation process of $X_t^x$.
\end{proof}

As an immediate consequence of \Cref{differentiable} and the representation of the Malliavin derivative $D_sX_t^x$, $0\leq s \leq t \leq T$, we get the following connection between the first variation process and the Malliavin derivative:

\begin{corollary}\label{derivativesRegular}
	Let $b \in C_b^{1,L}(\Rbb\times \Pcal_1(\Rbb))$. Then, for every $0 \leq s \leq t \leq T$,
	\begin{align}\label{relationMallFV}
		\partial_x X_t^x = D_sX_t^x \partial_x X_s^x + \int_s^t D_uX_t^x \partial_x b(u,y,\Pbb_{X_u^x})\vert_{y= X_u^x} du.
	\end{align}
\end{corollary}

Now let $b$ be a general drift coefficient that allows for a decomposition $b=\btilde + \bhat$ as in \eqref{formDrift} and is uniformly Lipschitz continuous in the third variable \eqref{eq:uniformLipschitzThird}. Let $(X_t^x)_{t\in[0,T]}$ be the corresponding strong solution of \eqref{mainMcKeanVlasov} ascertained by \Cref{strongSolution}. In order to extend \Cref{differentiable} we apply a compactness criterion to an approximating sequence of weakly differentiable mean-field SDEs. 
By standard approximation arguments there exists a sequence of approximating drift coefficients
\begin{align}\label{approximatingDrift}
	b_n:= \btilde_n + \bhat, \quad n\geq 1,
\end{align}
where $\btilde_n \in \Ccal_b^{1,L}( \Rbb \times \Pcal_1(\Rbb))$ with $\sup_{n\geq 1} \Vert \btilde_n\Vert_{\infty} \leq C < \infty$ such that $b_n \to b$ in \linebreak${(t,y,\mu) \in [0,T] \times \Rbb \times \Pcal_1(\Rbb)}$ a.e.~with respect to the Lebesgue measure. Furthermore, we denote $b_0 :=b$ and choose the approximating coefficients $b_n$ such that they fulfill the uniform Lipschitz continuity in the third variable \eqref{eq:uniformLipschitzThird} uniformly in $n\geq 0$. Under these conditions the corresponding mean-field SDEs, defined by
\begin{align}\label{approximatingMFSDEExp}
	dX_t^{n,x} = b_n(t,X_t^{n,x}, \Pbb_{X_t^{n,x}}) dt + dB_t, \quad X_0^{n,x} = x \in \Rbb, \quad t\in[0,T], \quad n\geq 1,
\end{align}
have unique strong solutions which are Malliavin differentiable by \Cref{strongSolution}. Likewise the strong solutions $\lbrace X^{n,x}\rbrace_{n\geq 1}$ are weakly differentiable with respect to the initial condition by \Cref{differentiable}. In the next step we verify that $(X_t^{n,x})_{t\in[0,T]}$ converges to $(X_t^x)_{t\in[0,T]}$ in $L^2(\Omega,\Fcal_t)$ as $n\to \infty$.

\begin{theorem}\label{convergenceApproximating}
	Suppose the drift coefficient $b$ is in the decomposable form \eqref{formDrift} and uniformly Lipschitz continuous in the third variable \eqref{eq:uniformLipschitzThird}. Let $(X_t^x)_{t\in[0,T]}$ be the unique strong solution of \eqref{mainMcKeanVlasov}. Furthermore, $\lbrace b_n \rbrace_{n\geq 1}$ is the approximating sequence of $b$ as defined in \eqref{approximatingDrift} and $(X_t^{n,x})_{t\in[0,T]}$, $n\geq 1$, the corresponding unique strong solutions of \eqref{approximatingMFSDEExp}. Then, there exists a subsequence $(n_k)\subset \Nbb$  such that 
\begin{align*}
	X_t^{n_k,x} \xrightarrow[k \to \infty]{L^2(\Omega, \Fcal_t)} X_t^x, \quad t\in[0,T].
\end{align*}
\end{theorem}
\begin{proof}
	In the case of SDEs it is shown in \cite{MeyerBrandisBanosDuedahlProske_ComputingDeltas}[Theorem A.4] that for every $t\in [0,T]$, the sequence $\lbrace X_t^{n,x} \rbrace_{n\geq 1}$ is relatively compact in $L^2(\Omega,\Fcal_t)$. The proof therein can be extended to the assumptions of \Cref{convergenceApproximating} and the case of mean-field SDEs due to \Cref{derivativeMalliavin}. Consequently, for every $t\in[0,T]$ there exists a subsequence $\lbrace n_k(t)\rbrace_{k\geq 1} \subset \Nbb$ such that $X_t^{n_k(t),x}$ converges to some $Y_t$ strongly in $L^2(\Omega,\Fcal_t)$. We need to show that the converging subsequence can be chosen independent of $t$. To this end we consider the Hida test function space $\Scal$ and the Hida distribution space $\Scal^*$ as defined in \Cref{HidaSpaces} and prove that $\lbrace t \mapsto X_t^{n,x} \rbrace_{n\geq 1}$ is relatively compact in $\Ccal([0,T];\Scal^*)$, which is well-defined since 
\begin{align*}
	\Scal \subset L^2(\Omega) \subset \Scal^*.
\end{align*}	 
In order to show this, we use \Cref{mitoma} and show instead that $\lbrace t \mapsto X_t^{n,x}[\phi] \rbrace_{n\geq 1}$ is relatively compact in $\Ccal([0,T];\Rbb)$ for any $\phi \in \Scal$, where $X_t^{n,x}[\phi] := \Ebb[X_t^{n,x} \phi]$. Since $X^{n,x}$ is a solution of \eqref{approximatingMFSDEExp}, using Cauchy-Schwarz' inequality and \Cref{boundsSolution} yields
\begin{align*}
	&\left\vert X_t^{n,x}[\phi] - X_s^{n,x}[\phi] \right\vert = \left\vert \Ebb[(X_t^{n,x} - X_s^{n,x})\phi] \right\vert \\
	& \quad = \left\vert \Ebb\left[ \left(\int_s^t b_n(u,X_u^{n,x}, \Pbb_{X_u^{n,x}}) du + B_t-B_s \right)\phi\right] \right\vert \\
	& \quad \leq \left( \int_s^t \Ebb[b_n(u,X_u^{n,x}, \Pbb_{X_u^{n,x}})^2]^{\frac{1}{2}} du  + |t-s| \right) \Vert \phi \Vert_{L^2(\Omega)} \leq C|t-s|.
\end{align*}
Hence, $\lbrace t \mapsto X_t^{n,x}[\phi] \rbrace_{n\geq 1}$ is relatively compact in $\Ccal([0,T];\Rbb)$ by \Cref{relativeCompactModulus}. Since $\phi$ was arbitrary, we have proven that $\lbrace t \mapsto X_t^{n,x} \rbrace_{n\geq 1}$ is relatively compact in $\Ccal([0,T];\Scal^*)$, i.e. there exists a subsequence $(n_k)_{k\geq 1}$ and $\lbrace t \mapsto Z_t \rbrace \in \Ccal([0,T];\Scal^*)$ such that
\begin{align}\label{convergenceFunction}
	\lbrace t \mapsto X_t^{n_k,x} \rbrace \xrightarrow[k\to\infty]{} \lbrace t \mapsto Z_t \rbrace
\end{align}
 in $\Ccal([0,T];\Scal^*)$. Furthermore, we have shown that for every $t\in[0,T]$ there exists a subsequence ${(n_{k_m}(t))_{m\geq 1}\subset (n_k)_{k\geq 1}}$ such that in $L^2(\Omega,\Fcal_t)$,
\begin{align*}
	X_t^{n_{k_m}(t),x} \xrightarrow[m\to\infty]{} Y_t.
\end{align*}
Note that for every $t\in[0,T]$, we get by \eqref{convergenceFunction}
\begin{align*}
	X_t^{n_{k_m}(t),x} \xrightarrow[m\to\infty]{} Z_t
\end{align*}
in $\Scal^*$. By uniqueness of the limit $Y_t = Z_t$ for every $t\in[0,T]$ and hence, the convergence in $L^2(\Omega,\Fcal_t)$ holds for the $t$ independent subsequence $(n_k)_{k\geq 1}$. \\
In the last step, which is deferred to the subsequent lemma, we show for all $t\in[0,T]$ that $X_t^{n,x}$ converges weakly in $L^2(\Omega,\Fcal_t)$ to the unique strong solution $\overline{X}_t^x$ of SDE
\begin{align}\label{approximatingMFSDEExp2}
	d\overline{X}_t^x = b(t,\overline{X}_t^x, \Pbb_{Y_t}) dt + dB_t, \quad \overline{X}_0^x = x \in \Rbb,\quad t\in[0,T].
\end{align}
Consequently, $X_t^{n,x}$ converges to $X_t^x$ in $L^2(\Omega, \Fcal_t)$. Indeed, we have shown that $X_t^{n,x}$ converges in $L^2(\Omega,\Fcal_t)$ to $Y_t$ for all $t\in[0,T]$. Moreover $X_t^{n,x}$ converges weakly in $L^2(\Omega, \Fcal_t)$ to $\overline{X}_t^x$ for all $t\in[0,T]$. Hence, by uniqueness of the limit, $Y_t \stackrel{d}{=} \overline{X}_t^x$ for all $t\in [0,T]$. Thus \eqref{approximatingMFSDEExp2} is identical to \eqref{mainMcKeanVlasov} and we can write $X=\overline{X}$, which shows \Cref{convergenceApproximating}.
\end{proof}

In the following we assume without loss of generality that the whole sequence $\{ X_t^{n,x} \}_{n\geq 1}$ converges to $X_t^{x}$ strongly in $L^2(\Omega,\Fcal_t)$ for every $t\in[0,T]$. Moreover, we define the weight function $\omega_T: \Rbb \to \Rbb$ by
\begin{align}\label{weightFunction}
	\omega_T(y) := \exp \left\lbrace - \frac{|y|^2}{4T} \right\rbrace, \quad y \in \Rbb.
\end{align}

\begin{lemma}\label{weakConvergence}
	Suppose the drift coefficient $b$ is in the decomposable form \eqref{formDrift} and uniformly Lipschitz continuous in the third variable \eqref{eq:uniformLipschitzThird}. Let $(X_t^x)_{t\in[0,T]}$ be the unique strong solution of \eqref{mainMcKeanVlasov}. Furthermore, $\lbrace b_n \rbrace_{n\geq 1}$ is the approximating sequence of $b$ as defined in \eqref{approximatingDrift} and $(X_t^{n,x})_{t\in[0,T]}$, $n\geq 1$, the corresponding unique strong solutions of \eqref{approximatingMFSDEExp}. Then, for every $t \in [0,T]$ and function $\phi \in L^{2p}(\Rbb; \omega_T)$ with $p:= \frac{1+\varepsilon}{\varepsilon}$, $\varepsilon>0$ sufficiently small with regard to \Cref{boundsSolution},
\begin{align*}
	\phi(X_t^{n,x}) \xrightarrow[n\to \infty]{} \phi(X_t^x)
\end{align*}
weakly in $L^2(\Omega,\Fcal_t)$.
\end{lemma}
\begin{proof}
As described in the proof of \Cref{convergenceApproximating} it suffices to show for all $t\in[0,T]$ that $\phi(X_t^{n,x})$ converges weakly to $\phi(\overline{X}_t^x)$, where $\overline{X}_t^x$ is the unique strong solution of SDE \eqref{approximatingMFSDEExp2}. This can be shown equivalently to \cite{MeyerBrandisBanosDuedahlProske_ComputingDeltas}[Lemma A.3]. First note that $\phi(X_t^{n,x}), \phi(\overline{X}_t^x) \in L^2(\Omega,\Fcal_t)$, $n\geq 0$. Hence, in order to show weak convergence it suffices to show that
\begin{align*}
	\Wcal(\phi(X_t^{n,x}))(f) \xrightarrow[n\to\infty]{} \Wcal(\phi(\overline{X}_t^x))(f),
\end{align*}
for every $f\in L^2([0,T])$. One can show by Hölder's inequality, inequality \eqref{exponentialInequality} and \Cref{boundsSolution} that
\begin{align*}
	&\left| \Wcal(\phi(X_t^{n,x}))(f) - \Wcal(\phi(\overline{X}_t^x))(f) \right| = \\
	&\lesssim \Ebb \left[ \left( \Ecal\left( \int_0^T b_n(s,B_s^x,\Pbb_{X_s^{n,x}}) + f(s) dB_s \right) - \Ecal\left( \int_0^T b(s,B_s^{x},\Pbb_{Y_s}) + f(s) dB_s \right) \right)^{q} \right]^{\frac{1}{q}} \\
	&\lesssim A_n,
\end{align*}
where $q:= \frac{2(1+\varepsilon)}{2+\varepsilon}$ and
\begin{align*}
	A_n &:= \Ebb \left[ \left(\int_0^T \left( b_n(s,B_s^{x},\Pbb_{X_s^{n,x}}) - b(s,B_s^{x},\Pbb_{Y_s}) \right) dB_s\right. \right.\\
	&\quad \left.\left.- \frac{1}{2} \int_0^T \left( (b_n(s,B_s^{x},\Pbb_{X_s^{n,x}})+f(s))^2 - (b(s,B_s^{x},\Pbb_{Y_s})+f(s))^2 \right) ds \right)^{2p}\right]^{\frac{1}{2p}}.
\end{align*}
Using Minkowski's inequality and Burkholder-Davis-Gundy's inequality yields
\begin{align*}
	A_n &\leq\Ebb \left[ \left|\int_0^T b_n(s,B_s^{x},\Pbb_{X_s^{n,x}}) - b(s,B_s^{x},\Pbb_{Y_s})  dB_s\right|^{2p}\right]^{\frac{1}{2p}} \\
	&\quad + \Ebb \left[ \left|\frac{1}{2} \int_0^T (b_n(s,B_s^{x},\Pbb_{X_s^{n,x}})+f(s))^2 - (b(s,B_s^{x},\Pbb_{Y_s})+f(s))^2 ds\right|^{2p} \right]^{\frac{1}{2p}} \\
	&\lesssim\Ebb \left[ \left(\int_0^T \left| b_n(s,B_s^{x},\Pbb_{X_s^{n,x}}) - b(s,B_s^{x},\Pbb_{Y_s}) \right|^{2} ds\right)^p \right]^{\frac{1}{2p}} \\
	&\quad + \Ebb \left[ \left( \int_0^T \left| (b_n(s,B_s^{x},\Pbb_{X_s^{n,x}})+f(s))^2 - (b(s,B_s^{x},\Pbb_{Y_s})+f(s))^2 \right| ds\right)^{2p} \right]^{\frac{1}{2p}} \\
	&=: D_n + E_n.
\end{align*}
Looking at the first summand, we see using the triangle inequality that
\begin{align*}
	D_n &= \Ebb \left[ \left(\int_0^T \left| b_n(s,B_s^{x},\Pbb_{X_s^{n,x}}) - b(s,B_s^{x},\Pbb_{Y_s}) \right|^{2} ds\right)^p \right]^{\frac{1}{2p}} \\
	&\leq \Ebb \left[ \left(\int_0^T \left| b_n(s,B_s^{x},\Pbb_{X_s^{n,x}}) - b_n(s,B_s^{x},\Pbb_{Y_s}) \right|^{2} ds\right)^p \right]^{\frac{1}{2p}}\\
	&\quad + \Ebb \left[ \left(\int_0^T \left| b_n(s,B_s^{x},\Pbb_{Y_s}) - b(s,B_s^{x},\Pbb_{Y_s}) \right|^{2} ds\right)^p \right]^{\frac{1}{2p}}.
\end{align*}
Since there exists a constant $C>0$ such that $(\mu \mapsto b_n(t,y,\mu)) \in \Lip_C(\Pcal_1(\Rbb))$ for all $n\geq 0$, $t\in[0,T]$, $y \in\Rbb$ and $X_s^{n,x} \xrightarrow[n\to\infty]{L^2(\Omega,\Fcal_s)} Y_s$ for all $s\in[0,T]$ by the proof of \Cref{convergenceApproximating}, we get by dominated convergence that $D_n$ converges to $0$ as $n\to\infty$. Equivalently one can show that also $E_n$ converges to $0$ as $n$ tends to infinity. Therefore $\left| \Wcal(\phi(X_t^{n,x}))(f) - \Wcal(\phi(\overline{X}_t^x))(f) \right|$ converges to $0$ as $n\to \infty$ and the claim holds.	
\end{proof}

The following lemma will be used in the application of the compactness argument in the proof of \Cref{solutionInSobolev}.

\begin{lemma}\label{pNormDerivative}
Let $\lbrace (X_t^{n,x})_{t\in[0,T]} \rbrace_{n\geq 1}$ be the unique strong solutions of \eqref{approximatingMFSDEExp}. Then, for any compact subset $K \subset \Rbb$ and $p\geq 2$,
\begin{align*}
	\sup_{n\geq 1} \sup_{t\in[0,T]} \esssup_{x\in K} \Ebb\left[ \left| \partial_x X_t^{n,x} \right|^p \right] \leq C,
\end{align*}
for some constant $C>0$.
\end{lemma}
\begin{proof}
By Theorem \ref{derivativesRegular}, we have
\begin{align}\label{eq:repPartialDerivative}
	\partial_x X_t^{n,x} = D_0X_t^{n,x} + \int_0^t D_uX_t^{n,x} \partial_x b_n(u,y,\Pbb_{X_u^{n,x}})\vert_{y=X_u^{n,x}} du.
\end{align}
Using \Cref{derivativeMalliavin} as well as Girsanov's theorem and Hölder's inequality with $q := \frac{1+\varepsilon}{\varepsilon}$, $\varepsilon>0$ sufficiently small with regard to \Cref{boundsSolution}, yields together with \Cref{boundLocalTime} that
\begin{align}\label{boundMalliavinPNorm}
\begin{split}
	\Ebb &\left[ \left| D_sX_t^{n,x} \right|^p \right] = \Ebb\left[ \exp\left\lbrace -p \int_s^t \int_{\Rbb} b_n(u,y,\Pbb_{X_u^{n,x}}) L^{X^x}(du,dy) \right\rbrace \right] \\
	&\lesssim \Ebb\left[ \exp\left\lbrace -qp \int_s^t \int_{\Rbb} b_n(u,y,\Pbb_{X_u^{n,x}}) L^{B^x}(du,dy) \right\rbrace \right]^{\frac{1}{q}} \leq C_1,
\end{split}
\end{align}
for some constant $C_1>0$ independent of $n\geq 0$, $x\in K$ and $s,t\in [0,T]$. Hence, we get for every $n\geq 1$ and almost every $x \in K$ with Minkowski's and Hölder's inequality using that $(\mu \mapsto b(t,y,\mu)) \in \Lip_C(\Pcal_1(\Rbb))$ for every $t\in[0,T]$ and $y\in\Rbb$ that

\begin{align}\label{eq:boundDerivativeFirstStage}
\begin{split}
	\Ebb&\left[ \left| \partial_x X_t^{n,x} \right|^p \right]^{\frac{1}{p}} = \Ebb\left[ \left| D_0X_t^{n,x} + \int_0^t D_uX_t^{n,x} \partial_x b_n(u,y,\Pbb_{X_u^{n,x}})\vert_{y=X_u^{n,x}} du \right|^p \right]^{\frac{1}{p}} \\
	&\leq \sup_{0\leq u \leq T} \Ebb\left[ \left| D_uX_t^{n,x} \right|^{2p} \right]^{\frac{1}{2p}} \left( 1 + \Ebb\left[ \left( \int_0^t \left| \partial_x b_n\left(u,y,\Pbb_{X_u^{n,x}}\right) \right|_{y=X_u^{n,x}} du \right)^{2p} \right]^{\frac{1}{2p}} \right)\\
	&\lesssim 1 + \Ebb\left[ \left( \int_0^t \left| \lim_{x_0\to x} \frac{b_n\left(u,X_u^{n,x},\Pbb_{X_u^{n,x}} \right) - b_n\left(u,X_u^{n,x},\Pbb_{X_u^{n,x_0}} \right)}{|x-x_0|} \right| du \right)^{2p} \right]^{\frac{1}{2p}} \\
	&\lesssim 1 + \liminf_{x_0\to x} \frac{1}{|x-x_0|} \int_0^t \Kcal\left(\Pbb_{X_u^{n,x}}, \Pbb_{X_u^{n,x_0}} \right) du.
	\end{split}
\end{align}
Denote by $\overline{\text{conv}(K)}$ the closed convex hull of $K$ and note that $\overline{\text{conv}(K)}$ is again a compact set. Moreover, we can bound the Kantorovich metric of $\Pbb_{X_u^{n,x}}$ and $\Pbb_{X_u^{n,x_0}}$ for arbitrary $x,x_0 \in \overline{\text{conv}(K)}$ by using the second fundamental theorem of calculus and representation \eqref{representationFV}:
\begin{align}\label{eq:lipschitzExpecation}
\begin{split}
	\Kcal&\left(\Pbb_{X_u^{n,x}},\Pbb_{X_u^{n,x_0}} \right) \leq \Ebb\left[ \left| \int_0^u b_n(s,X_s^{n,x},\Pbb_{X_s^{n,x}}) - b_n(s,X_s^{n,x_0},\Pbb_{X_s^{n,x_0}}) ds \right|\right] \\
	&= |x-x_0| \Ebb\left[ \left|  \int_0^u \int_0^1 \partial_2 b_n\left(s,X_s^{n,x+\tau (x_0-x)}, \Pbb_{X_s^{n,x+\tau (x_0-x)}}\right) \partial_{\tau} X_s^{n,x+ \tau (x_0-x)} \right.\right.\\
		&\quad \left.\left.+ \partial_{\tau} b_n\left(s,z, \Pbb_{X_s^{n,x+\tau (x_0-x)}}\right)\vert_{z=X_s^{n,x+\tau (x_0-x)}} d\tau ds \right|\right] \\
	&\leq |x-x_0| \int_0^1 \Ebb\left[ \left|  \int_0^u  \partial_2 b_n\left(s,X_s^{n,x+\tau (x_0-x)}, \Pbb_{X_s^{n,x+\tau (x_0-x)}}\right) \partial_{\tau} X_s^{n,x+ \tau (x_0-x)}\right.\right. \\
	&\quad \left.\left.+ \partial_{\tau} b_n\left(s,z, \Pbb_{X_s^{n,x+\tau (x_0-x)}}\right)\vert_{z=X_s^{n,x+\tau (x_0-x)}}  ds \right|\right] d\tau \\ 
	&= |x-x_0| \int_0^1 \Ebb\left[ \left| \partial_{\tau} X_u^{n,x+ \tau (x_0-x)}  - (1-\tau) \right|\right] d\tau \\
	&\lesssim |x-x_0| + |x-x_0| \esssup_{x\in \overline{\text{conv}(K)}} \Ebb\left[ \left| \partial_x X_u^{n,x} \right|\right].
\end{split}
\end{align}
Putting all together we can find a constant $C_2>0$ independent of $n\geq 1$, $t\in[0,T]$ and $x\in \overline{\text{conv}(K)}$ such that 
\begin{align*}
	\esssup_{x\in \overline{\text{conv}(K)}} \Ebb\left[ \left| \partial_x X_t^{n,x} \right|^p\right]^{\frac{1}{p}} \leq C_2 + C_2 \int_0^t \esssup_{x\in \overline{\text{conv}(K)}} \Ebb\left[ \left| \partial_x X_u^{n,x} \right|^p\right]^{\frac{1}{p}} du.
\end{align*}
Note that by \eqref{eq:boundDerivativeFirstStage} and \eqref{measureLipschitz} we can find constants $C_3(n),C_4(n)>0$ for every $n\geq 1$ independent of $t\in[0,T]$ and $x\in \overline{\text{conv}(K)}$ such that 
\begin{align*}
	\Ebb\left[ \left| \partial_x X_t^{n,x} \right|^p\right]^{\frac{1}{p}} \leq C_3(n)\left(1 + \liminf_{x_0\to x} \frac{1}{|x-x_0|} \int_0^t \Kcal\left(\Pbb_{X_u^{n,x}}, \Pbb_{X_u^{n,x_0}} \right) du \right) \leq C_4(n) < \infty.
\end{align*}
Hence,  $t\mapsto \esssup_{x\in \overline{\text{conv}(K)}} \Ebb\left[ \left| \partial_x X_t^{n,x} \right|^p\right]^{\frac{1}{p}}$ is integrable over $[0,T]$. Since it is also Borel measurable, we can apply Jones' generalization of Grönwall's inequality \linebreak \cite{Jones_FundamentalInequalities}[Lemma 5] to get
\begin{align*}
	\esssup_{x\in K} \Ebb\left[ \left| \partial_x X_t^{n,x} \right|^p\right]^{\frac{1}{p}} \leq \esssup_{x\in \overline{\text{conv}(K)}} \Ebb\left[ \left| \partial_x X_t^{n,x} \right|^p\right]^{\frac{1}{p}} \leq C _2+ C_2^2 \int_0^t e^{C_2(t-s)} ds < \infty.
\end{align*}
\end{proof}

Before we prove weak differentiability of $X^x$ we recall the definition of the Sobolev space $W^{1,2}(U)$.

\begin{definition}
Let $U \subset \Rbb$ be an open and bounded subset. The Sobolev space $W^{1,2}(U)$ is defined as the set of functions $u:\Rbb \to \Rbb$, $u\in L^2(U)$, such that its weak derivative belongs to $L^2(U)$. Furthermore, the Sobolev space is endowed with the norm
\begin{align*}
	\Vert u \Vert_{W^{1,2}(U)} = \Vert u \Vert_{L^2(U)} + \Vert u' \Vert_{L^2(U)},
\end{align*}
where $u'$ is the weak derivative of $u \in W^{1,2}(U)$. We say a stochastic process $X$ is Sobolev differentiable in $U$, if for all $t\in[0,T]$, $X_t^{\cdot}$ belongs $\Pbb$-a.s. to $W^{1,2}(U)$.
\end{definition}

\begin{theorem}\label{solutionInSobolev}
Suppose the drift coefficient $b$ is in the decomposable form \eqref{formDrift} and uniformly Lipschitz continuous in the third variable \eqref{eq:uniformLipschitzThird}. Let $(X_t^x)_{t\in[0,T]}$ be the unique strong solution of \eqref{mainMcKeanVlasov} and $U\subset \Rbb$ be an open and bounded subset. Then for every $t\in[0,T]$,
\begin{align*}
	(x \mapsto X_t^x) \in L^2\left(\Omega, W^{1,2}(U)\right).
\end{align*}
\end{theorem}
\begin{proof}
	Let $(X_t^{n,x})_{t\in[0,T]}$ be the unique strong solutions of \eqref{approximatingMFSDEExp}. The main idea of this proof is to show that  $\left\lbrace X_t^n \right\rbrace_{n\geq 1}$ is weakly relatively compact in $L^2(\Omega,W^{1,2}(U))$ and identifying the relative limit $Y:=\lim_{k\to\infty} X^{n_k}$ in $L^2(\Omega,W^{1,2}(U))$ with $X$, where $\lbrace n_k \rbrace_{k\geq 1}$ is a suitable sub-sequence.\par
	Due to \Cref{boundsSolution} and \Cref{pNormDerivative}
	\begin{align*}
		\sup_{n\geq 1} \Ebb\left[ \Vert X_t^{n,x} \Vert_{W^{1,2}(U)}^2 \right] < \infty,
	\end{align*}
	and thus, the sequence $X_t^{n,x}$ is weakly relatively compact in $L^2(\Omega,W^{1,2}(U))$, see e.g. \cite{Leoni}[Theorem 10.44]. Consequently, there exists a sub-sequence $n_k$, $k\geq 0$ such that $X_t^{n_k,x}$ converges weakly to some $Y_t \in L^2(\Omega, W^{1,2}(U))$ as $k\to \infty$. 
	Let $\phi \in \Ccal_0^{\infty}(U)$ be an arbitrary test-function and denote by $\phi'$ if well-defined its first derivative. Define
	\begin{align*}
		\left\langle X_t^n, \phi \right\rangle := \int_U X_t^{n,x} \phi(x) dx.
	\end{align*}
	Then for all measurable sets $A \in \Fcal$ and $t\in[0,T]$ we get by \Cref{boundsSolution} that
\begin{align*}
	\Ebb \left[ \mathbbm{1}_A \langle X_t^{n}-X_t, \phi' \rangle \right] \leq \Vert \phi' \Vert_{L^2(U)} \vert U \vert^{\frac{1}{2}} \sup_{x\in\overline{U}}\Ebb \left[ \mathbbm{1}_A \vert X_t^{n,x}-X_t^x \vert^2 \right]^{\frac{1}{2}} < \infty,
\end{align*}
where $\overline{U}$ is the closure of $U$, and consequently by \Cref{convergenceApproximating} we get that \linebreak $\lim_{n\to\infty} \Ebb \left[ \mathbbm{1}_A \langle X_t^{n}-X_t, \phi' \rangle \right] = 0.$ Therefore,
\begin{align*}
	\Ebb[\mathbbm{1}_A \langle X_t, \phi'\rangle] = \lim_{k\to\infty} \Ebb[\mathbbm{1}_A \langle X_t^{n_k}, \phi' \rangle] = - \lim_{k\to\infty} \Ebb\left[\mathbbm{1}_A \left\langle \partial_x X_t^{n_k},\phi \right\rangle\right] = - \Ebb\left[\mathbbm{1}_A \left\langle \partial_x Y_t,\phi \right\rangle\right].
\end{align*}
Thus,
\begin{align}\label{eq:weakDerivativeExists}
	\Pbb\text{-a.s.} \quad \left\langle X_t, \phi' \right\rangle = - \left\langle \partial_x  Y_t, \phi \right\rangle.
\end{align}
Finally, we have to show as in \cite{MeyerBrandisBanosDuedahlProske_ComputingDeltas}[Theorem 3.4] that there exists a measurable set $\Omega_0 \subset \Omega$ with full measure such that $X_t^{\cdot}$ has a weak derivative on this subset. To this end, choose a sequence $\lbrace \phi_n \rbrace_{n\geq 1} \subset \Ccal_0^{\infty}(\Rbb)$ dense in $W^{1,2}(U)$ and a measurable subset $\Omega_n \subset \Omega$ with full measure such that \eqref{eq:weakDerivativeExists} holds on $\Omega_n$ with $\phi$ replaced by $\phi_n$. Then $\Omega_0 := \bigcap_{n\geq 1} \Omega_n$ satisfies the desired property.
\end{proof}

As a direct consequence of Theorem~\ref{solutionInSobolev} we are able to extend \Cref{pNormDerivative} to include also $\partial_x X_t^x$:

\begin{corollary}\label{derivativeInitialLpBound}
	Suppose the drift coefficient $b$ is in the decomposable form \eqref{formDrift} and uniformly Lipschitz continuous in the third variable \eqref{eq:uniformLipschitzThird}. Let $(X_t^x)_{t\in[0,T]}$ be the unique strong solution of \eqref{mainMcKeanVlasov}. Then for any compact subset $K \subset \Rbb$ and $p\geq 1$, there exists a constant $C>0$ such that
	\begin{align*}
		\sup_{t\in[0,T]} \esssup_{x\in K} \Ebb\left[ \left( \partial_x X_t^x \right)^p \right] \leq C.
	\end{align*}
\end{corollary}
\begin{proof}
	The proof follows by \Cref{pNormDerivative} and the application of Fatou's lemma:
	\begin{align*}
		\Ebb\left[ \left( \partial_x X_t^x \right)^p \right] \leq \liminf_{n\to\infty} \Ebb\left[ \left( \partial_x X_t^{n,x} \right)^p \right] \leq  C.
	\end{align*}
\end{proof}

We conclude this section by Proposition~\ref{derivativeInitial} that generalizes the probabilistic representation \eqref{representationFV} of the first variation process $\left( \partial_x X_t^x \right)_{t\in[0,T]}$ and the connection to the Malliavin derivative given in \Cref{derivativesRegular} to irregular drift coefficients. To this end we first verify the weak differentiability of the function $\left(x \mapsto  b\left(t,y,\Pbb_{X_t^x}\right) \right)$ in the next proposition.

\begin{proposition}\label{driftInSobolev}
	Suppose the drift coefficient $b$ is in the decomposable form \eqref{formDrift} and uniformly Lipschitz continuous in the third variable \eqref{eq:uniformLipschitzThird}. Let $(X_t^x)_{t\in[0,T]}$ be the unique strong solution of \eqref{mainMcKeanVlasov} and $U\subset \Rbb$ be an open and bounded subset. Then for every $1<p<\infty$, $t\in[0,T]$ and $y\in \Rbb$,
\begin{align*}
	\left(x \mapsto b\left(t,y,\Pbb_{X_t^x}\right) \right) \in W^{1,p}(U).
\end{align*}
\end{proposition}
\begin{proof}
	Let $\lbrace b_n \rbrace_{n\geq 1}$ be the approximating sequence of $b$ as defined in \eqref{approximatingDrift} and $(X_t^{n,x})_{t\in[0,T]}$, $n\geq 1$, the corresponding unique strong solutions of \eqref{approximatingMFSDEExp}. For notational simplicity we define $b_n(x) := b_n\left(t,y,\Pbb_{X_t^{n,x}}\right)$ for every $n\geq 0$. We proceed similar to the proof of \Cref{solutionInSobolev} and thus start by showing that $\lbrace b_n \rbrace_{n\geq 1}$ is weakly relatively compact in $W^{1,p}(U)$. Due to \Cref{boundsSolution} and the proof of \Cref{pNormDerivative}
	\begin{align*}
		\sup_{n\geq 1} \Vert b_n \Vert_{W^{1,p}(U)} < \infty.
	\end{align*}
	Hence, $\lbrace b_n \rbrace$ is bounded in $W^{1,p}(U)$ and thus weakly relatively compact by \linebreak\cite{Leoni}[Theorem 10.44]. Therefore, we can find a sub-sequence $\lbrace n_k \rbrace_{k\geq 1}$ and $g\in W^{1,p}(U)$ such that $b_{n_k}$ converges weakly to $g$ as $k\to\infty$.\par
	Let $\phi \in \Ccal_0^{\infty}(U)$ be an arbitrary test-function and denote by $\phi'$ if well-defined its first derivative. Define
	\begin{align*}
		\left\langle b_n, \phi \right\rangle := \int_U b_n(x) \phi(x) dx.
	\end{align*}
	Due to \Cref{boundsSolution}
	\begin{align*}
		\langle b_n - b , \phi' \rangle \leq \Vert \phi' \Vert_{L^p(U)} \vert U \vert^{\frac{1}{p}} \sup_{x\in\overline{U}} \vert b_n(x) - b(x) \vert < \infty,
	\end{align*}
	where $\overline{U}$ is the closure of $U$, and since by \Cref{convergenceApproximating}
	\begin{align*}
		&\left\vert b_n\left(t,y,\Pbb_{X_t^{n,x}} \right) - b\left(t,y,\Pbb_{X_t^x} \right) \right\vert \\
		&\quad \leq \left\vert b_n\left(t,y,\Pbb_{X_t^{n,x}} \right) - b_n\left(t,y,\Pbb_{X_t^x} \right) \right\vert + \left\vert b_n\left(t,y,\Pbb_{X_t^x} \right) - b\left(t,y,\Pbb_{X_t^x} \right) \right\vert \\
		&\quad \leq C \Kcal\left(\Pbb_{X_t^{n,x}}, \Pbb_{X_t^x} \right) + \left\vert b_n\left(t,y,\Pbb_{X_t^x} \right) - b\left(t,y,\Pbb_{X_t^x} \right) \right\vert \xrightarrow[n\to\infty]{} 0,
	\end{align*}
	we get $\lim_{n\to\infty} \langle b_n-b,\phi' \rangle = 0.$ Thus,
\begin{align*}
	\langle b, \phi'\rangle = \lim_{k\to\infty} \langle b_{n_k}, \phi' \rangle = - \lim_{k\to\infty} \left\langle b_{n_k}',\phi \right\rangle = - \left\langle g',\phi \right\rangle,
\end{align*}
where $b_{n_k}'$ and $g'$ are the first variation processes of $b_{n_k}$ and $g$, respectively.
\end{proof}


\begin{proposition}\label{derivativeInitial}
	Suppose the drift coefficient $b$ is in the decomposable form \eqref{formDrift} and uniformly Lipschitz continuous in the third variable \eqref{eq:uniformLipschitzThird}. For almost all $x \in \Rbb$ the first variation process (in the Sobolev sense) of the unique strong solution $(X_t^x)_{t\in[0,T]}$ of the mean-field SDE \eqref{mainMcKeanVlasov} has $dt\otimes d\Pbb$ almost surely the representation
	\begin{align}\label{weakDerivative}
	\begin{split}
		\partial_x X_t^x &= \exp \left\lbrace - \int_0^t \int_{\Rbb} b\left(u,y,\Pbb_{X_u^x}\right) L^{X^x}(du,dy) \right\rbrace \\
		&\quad + \int_0^t \exp \left\lbrace - \int_s^t \int_{\Rbb} b\left(u,y,\Pbb_{X_u^x}\right) L^{X^x}(du,dy) \right\rbrace \partial_x b\left(s,y,\Pbb_{X_s^x}\right)\vert_{y=X_s^x} ds.
	\end{split}
	\end{align}
Furthermore, for $s,t \in [0,T]$, $s\leq t$, the following relationship with the Malliavin Derivative holds:
	\begin{align}\label{relationDerivatives}
		\partial_x X_t^x = D_sX_t^x \partial_x X_s^x + \int_s^t D_uX_t^x \partial_x b\left(u,y,\Pbb_{X_u^x}\right)\vert_{y=X_u^x} du.
	\end{align}
\end{proposition}
\begin{proof}
	Let $(b_n)_{n\geq 1}$ be the approximating sequence of $b$ as defined in \eqref{approximatingDrift} and \linebreak$(X_t^{n,x})_{t\in[0,T]}$ be the corresponding unique strong solutions of \eqref{approximatingMFSDEExp}. We define for $n\geq 0$
	\begin{align*}
		\Psi_n &:= \exp \left\lbrace -\int_0^t\int_{\Rbb} b_n\left(u,y,\Pbb_{X_u^{n,x}}\right) L^{X^{n,x}}(du,dy) \right\rbrace \\
		&\quad + \int_0^t \exp \left\lbrace -\int_s^t\int_{\Rbb} b_n\left(u,y,\Pbb_{X_u^{n,x}}\right) L^{X^{n,x}}(du,dy) \right\rbrace \partial_x b_n\left(s,y,\Pbb_{X_s^{n,x}}\right)\vert_{y=X_s^{n,x}} ds.
	\end{align*}
	For every $t\in[0,T]$ the sequence $\lbrace X_t^{n,x} \rbrace_{n\geq 1}$ converges weakly in $L^2(\Omega,W^{1,2}(U))$ to $X_t^x$ by the proof of \Cref{solutionInSobolev}. Hence, it suffices to show for every $f\in L^2([0,T])$ and $g\in C_0^{\infty}(U)$ that
	\begin{align*}
		\left\langle \Wcal\left( \Psi_n - \Psi_0 \right)(f), g \right\rangle \xrightarrow[n\to\infty]{} 0.
	\end{align*}
	Define for every $n\geq 0$
	\begin{align*}
		L_n(s,t,x) &:= \exp\left\lbrace  -\int_s^t\int_{\Rbb} b_n\left(u,y,\Pbb_{X_s^{n,x}}\right) L^{B^x}(du,dy) \right\rbrace, \text{ and}\\
		\Ecal_n(x) &:= \Ecal\left( \int_0^T b_n\left(u,B_u^x,\Pbb_{X_s^{n,x}}\right) +f(u) dB_u\right).
	\end{align*}
	Applying Girsanov's theorem and Minkowski's inequality yields 
	\begin{align*}
		&\left\langle \Wcal\left( \Psi_n- \Psi_0 \right)(f), g \right\rangle \\
		&\quad \leq \int_U g(x) \Ebb\left[\left| L_n(0,t,x) - L_0(0,t,x) \right| \Ecal_n(x) \right] dx \\
		&\qquad + \int_U g(x) \Ebb\left[\left|\Ecal_n(x) - \Ecal_0(x) \right| L_0(0,t,x) \right] dx\\
		&\qquad + \int_U  \int_0^t g(x)\Ebb \left[\left| L_n(s,t,x) - L_0(s,t,x) \right| \left| \partial_x b_n\left(s,y,\Pbb_{X_s^{n,x}}\right) \right|_{y=B_s^x}  \Ecal_n(x) \right] ds dx \\
		&\qquad + \int_U  \int_0^t g(x) \Ebb \left[\left| \Ecal_n(x) - \Ecal_0(x) \right| L_0(s,t,x) \left| \partial_x b_n\left(s,y,\Pbb_{X_s^{n,x}}\right) \right|_{y=B_s^x} \right] ds dx\\
		&\qquad + \int_U \int_0^t g(x) \Ebb \left[ \left| \partial_x b_n\left(s,y,\Pbb_{X_s^{n,x}}\right) - \partial_x b\left(s,y,\Pbb_{X_s^x}\right) \right|_{y=B_s^x}  L_0(s,t,x) \Ecal_0(x) \right] ds dx.
	\end{align*}
	Note that for any $1 < p < \infty$,
	\begin{align}\label{eq:partialBounded}
	\sup_{n\geq 0} \sup_{s\in[0,T]} \esssup_{x\in U} \Ebb \left[ \left| \partial_x b_n\left(s,y,\Pbb_{X_s^{n,x}}\right)\vert_{y=B_s^x} \right|^p \right]  < \infty,
	\end{align}
	due to \Cref{derivativeInitialLpBound} and the proof of \Cref{pNormDerivative}. Hence, we get by Hölder's inequality, \Cref{boundsSolution}, and \Cref{boundLocalTime} that for $q:= \frac{2(1+\varepsilon)}{2+\varepsilon}$ and $p:= \frac{2(1+\varepsilon)}{\varepsilon}$, where $\varepsilon>0$ is sufficiently small with regard to \Cref{boundsSolution},
	\begin{align*}
		&\left\langle \Wcal\left( \Psi_n- \Psi_0 \right)(f), g \right\rangle \\
		&\quad \lesssim \int_U g(x) \left( \sup_{s,t\in[0,T]} \Ebb \left[\left| L_n(s,t,x) - L_0(s,t,x) \right|^p \right]^{\frac{1}{p}} + \Ebb\left[\left|\Ecal_n(x) - \Ecal_0(x) \right|^{q} \right]^{\frac{1}{q}} \right) dx\\
		&\qquad + \int_U \int_0^t g(x) \Ebb \left[ \left| \partial_x b_n\left(s,y,\Pbb_{X_s^{n,x}}\right) - \partial_x b\left(s,y,\Pbb_{X_s^x}\right) \right|_{y=B_s^x}^p\right]^{\frac{1}{p}} ds dx.
	\end{align*}
	 The first two summands converge due to \Cref{uniformConvergenceEcal}, \Cref{uniformConvergenceLocalTime}, and dominated convergence. For the third summand we use that $\left(x \mapsto b\left(t,y,\Pbb_{X_t^x}\right) \right) \in W^{1,p}(U)$. Consequently, by dominated convergence and \cite{Ziemer_weaklyDifferentiableFunctions}[Lemma 2.1.3] we get that
	 \begin{align*}
	 	\int_U \int_0^t g(x) \Ebb \left[ \left| \partial_x b_n\left(s,y,\Pbb_{X_s^{n,x}}\right) - \partial_x b\left(s,y,\Pbb_{X_s^x}\right) \right|_{y=B_s^x}^p\right]^{\frac{1}{p}} ds dx \xrightarrow[n\to\infty]{} 0.
	 \end{align*}
\end{proof}

\section{Bismut-Elworthy-Li formula}\label{SBEL}

In this section we turn our attention to finding a Bismut-Elworthy-Li type formula, i.e.~ with the help of \Cref{derivativeInitial} we give a probabilistic representation of type \eqref{BismutGeneral} for $\partial_x \Ebb[\Phi(X_T^x)]$ for functions $\Phi$ merely satisfying some integrability condition. The following lemma prepares the grounds for the main result in Theorem~\ref{mainTheoremDelta}.

\begin{lemma}\label{repDelta2}
	Suppose the drift coefficient $b$ is in the decomposable form \eqref{formDrift} and uniformly Lipschitz continuous in the third variable \eqref{eq:uniformLipschitzThird}. Let $(X_t^x)_{t\in[0,T]}$ be the unique strong solution of the corresponding mean-field SDE \eqref{mainMcKeanVlasov} and $U\subset \Rbb$ be an open and bounded subset. Furthermore, consider the functional $\Phi \in \Ccal_b^{1,1}(\Rbb)$. Then for every $t\in[0,T]$ and $1<p<\infty$,
\begin{align*}
	\left(x \mapsto \Ebb\left[ \Phi(X_t^x) \right] \right) \in W^{1,p}(U).
\end{align*}
Moreover, for almost all $x \in U$
\begin{align}\label{eq:repDerivativeExpectation}
	\partial_x \Ebb\left[ \Phi(X_t^x) \right] = \Ebb\left[ \Phi'(X_t^x) \partial_x X_t^x \right],
\end{align}
where $\Phi'$ denotes the first derivative of $\Phi$.
\end{lemma}
\begin{proof}
	It is readily seen that $\left( x \mapsto \Ebb[X_t^x] \right) \in \Lip_{C_1}(U,\Rbb)$ for some constant $C_1>0$ due to \eqref{eq:lipschitzExpecation} and \Cref{convergenceApproximating}. Therefore, we get with the assumptions on the functional $\Phi$ that there exists a constant $C_2>0$ such that $\left( x \mapsto \Ebb[\Phi(X_t^x)] \right) \in \Lip_{C_2}(U,\Rbb)$. Hence, $\Ebb[\Phi(X_t^x)]$ is almost everywhere and weakly differentiable on $U$ and for almost all $x \in U$
	\begin{align*}
		\partial_x \Ebb[\Phi(X_t^x)] &= \lim_{h\to 0} \frac{\Ebb[\Phi(X_t^{x+h})] - \Ebb[\Phi(X_t^x)]}{h} = \Ebb\left[ \lim_{h\to 0} \frac{\Phi(X_t^{x+h}) - \Phi(X_t^x)}{h} \right] \\
		&= \Ebb\left[ \Phi'(X_t^x) \partial_x X_t^x \right],
	\end{align*}
	where we used dominated convergence and the chain rule. Finally, we can conclude directly from \eqref{eq:repDerivativeExpectation} using \Cref{derivativeInitialLpBound} and the boundedness of $\Phi'$ that \linebreak $\left(x \mapsto \Ebb\left[ \Phi(X_t^x) \right] \right) \in W^{1,p}(U)$ for every $1<p<\infty$.
\end{proof}

\begin{theorem}\label{mainTheoremDelta}
	Suppose the drift coefficient $b$ is in the decomposable form \eqref{formDrift} and uniformly Lipschitz continuous in the third variable \eqref{eq:uniformLipschitzThird}. Let $(X_t^x)_{t\in[0,T]}$ be the unique strong solution of the corresponding mean-field SDE \eqref{mainMcKeanVlasov}, $K\subset \Rbb$ be a compact subset and $\Phi \in L^{2p}(\Rbb;\omega_T)$, where $p:= \frac{1+\varepsilon}{\varepsilon}$, $\varepsilon>0$ sufficiently small with regard to \Cref{boundsSolution}, and $\omega_T$ is as defined in \eqref{weightFunction}. Then, for every open subset $U \subset K$, $t\in[0,T]$ and $1<q<\infty$,
	\begin{align*}
		\left(x \mapsto \Ebb\left[ \Phi(X_t^x) \right] \right) \in W^{1,q}(U),
	\end{align*}
	and for almost all $x \in K$
	\begin{align}\label{Delta}
		\partial_x \Ebb[\Phi(X_T^x)] = \Ebb \left[ \Phi(X_T^x) \left( \int_0^T a(s) \partial_x X_s^x + \partial_x b\left(s,y,\Pbb_{X_s^x}\right)\vert_{y=X_s^x} \int_0^s a(u) du dB_s \right) \right],
	\end{align}
	where $\partial_x X_s^x$ is given in \eqref{weakDerivative} and $a: \Rbb  \to \Rbb$ is any bounded, measurable function such that
	\begin{align*}
		\int_0^T a(s) ds = 1.
	\end{align*}
\end{theorem}

\begin{remark}
	Note that in the case of an SDE the derivative \eqref{Delta} collapses to the representation
	\begin{align*}
		\Ebb \left[ \Phi(X_T^x) \int_0^T a(s) \partial_x X_s^x dB_s \right]
	\end{align*}
	established in \cite{MeyerBrandisBanosDuedahlProske_ComputingDeltas}, where the first variation process $\partial_x X^x$ has the representation
	\begin{align*}
		\partial_x X_t^x = \exp \left\lbrace - \int_0^t \int_{\Rbb} b(u,y) L^{X^x}(du,dy) \right\rbrace.
	\end{align*}
	Hence, one can speak of a derivative free representation. Regarding mean-field SDEs, the derivative $\partial_x b\left(s,y,\Pbb_{X_s^x}\right)$ still appears in the representation of $\partial_x X^x$.
\end{remark}

\begin{remark}
	In \cite{Banos_Bismut} the Bismut-Elworthy-Li formula \eqref{BismutGeneral} contained Skorohod integration. Here we find an adapted integrand and establish a representation only involving It\^{o} integration. Replacing Skorohod by It\^{o} integration enables a numerical simulation of $\partial_x \Ebb \left[ \Phi(X_T^x) \right]$ through the use of representation \eqref{Delta}.
\end{remark}

\begin{remark}\label{rem:regBismut}
	In \cite{Bauer_RegularityOfMFSDE} we show that for the special case of mean-field SDEs of type \eqref{eq:specialMFSDEGeneral}, the expectation functional $\Ebb[\Phi(X_t^x)]$ is even continuously differentiable in $x$ for irregular drift coefficients under certain additional assumptions on $\hat{b}$ and $\varphi$.
\end{remark}

\begin{proof}[Proof of \Cref{mainTheoremDelta}]
	We start by showing the result for $\Phi \in \Ccal_b^{1,1}(\Rbb)$. In this case the derivative $\partial_x \Ebb[\Phi(X_T^x)]$ exists by \Cref{repDelta2} and admits representation \eqref{eq:repDerivativeExpectation}. Furthermore, by \eqref{relationDerivatives} for any $s \leq T$,
	\begin{align*}
		\partial_x X_T^x = D_sX_T^x \partial_x X_s^x + \int_s^T D_uX_T^x \partial_x b\left(u,y,\Pbb_{X_u^x}\right)\vert_{y=X_u^x} du.
	\end{align*}
	Recall that $D_sX_T^x=0$ for $s\geq T$. Thus for any bounded function $a:\Rbb \to \Rbb$ with $\int_0^T a(s) ds = 1$,
	\begin{align*}
		\partial_x X_T^x &= \int_0^T a(s) \left( D_s X_T^x \partial_x X_s^x + \int_s^T D_uX_T^x \partial_x b\left(u,y,\Pbb_{X_u^x}\right)\vert_{y=X_u^x} du \right) ds \\
		&= \int_0^T a(s) D_s X_T^x \partial_x X_s^x ds  + \int_0^T \int_s^T a(s) D_uX_T^x \partial_x b\left(u,y,\Pbb_{X_u^x}\right)\vert_{y=X_u^x} du ds.
	\end{align*}
	We look at each summand individually starting with the first one. Since $\Phi \in \Ccal_b^{1,1}(\Rbb)$, $\Phi(X_T^x)$ is Malliavin differentiable and
	\begin{align*}
		\Ebb \left[ \Phi'(X_T^x) \int_0^T a(s) D_s X_T^x \partial_x X_s^x ds \right] = \Ebb \left[ \int_0^T a(s) D_s \Phi(X_T^x) \partial_x X_s^x ds \right].
	\end{align*}
	Due to the fact that $s \mapsto a(s) \partial_x X_s^x$ is an adapted process satisfying
	\begin{align*}
		\Ebb \left[ \int_0^T \left( a(s) \partial_x X_s^x \right)^2 ds \right] < \infty
	\end{align*}
	by \Cref{derivativeInitialLpBound}, we can apply the duality formula \cite{ProskeDiNunnoOksendal_MalliavinCalculus}[Corollary 4.4] and get
	\begin{align*}
		\Ebb \left[ \int_0^T a(s) D_s \Phi(X_T^x) \partial_x X_s^x ds \right] = \Ebb \left[ \Phi(X_T^x) \int_0^T a(s) \partial_x X_s^x dB_s \right].
	\end{align*}
	For the second summand note that by \eqref{boundMalliavinPNorm} and the proof of \Cref{pNormDerivative}
	\begin{align*}
		\sup_{u,s\in[0,T]} \Ebb\left[ \left| \Phi'(X_T^x) a(s) D_uX_T^x \partial_x b\left(u,y,\Pbb_{X_u^x}\right)\vert_{y=X_u^x} \right| \right] < \infty.
	\end{align*}
	Hence, the integral 
	\begin{align*}
		\int_0^T \int_0^T \Ebb \left[ \left|\Phi'(X_T^x) a(s) D_uX_T^x \partial_x b\left(u,y,\Pbb_{X_u^x}\right)\vert_{y=X_u^x} \right| \right] du ds
	\end{align*}
	exists and is finite by Tonelli's Theorem. Consequently, we can	 interchange the order of integration to deduce
	\begin{align*}
		&\Ebb\left[ \Phi'(X_T^x) \int_0^T \int_s^T a(s) D_uX_T^x \partial_x b\left(u,y,\Pbb_{X_u^x}\right)\vert_{y=X_u^x} du ds \right] \\
		&\quad = \Ebb\left[ \int_0^T D_u\Phi(X_T^x) \partial_x b\left(u,y,\Pbb_{X_u^x}\right)\vert_{y=X_u^x} \int_0^u a(s) ds du \right].
	\end{align*}
	Furthermore, $u \mapsto \partial_x b\left(u,y,\Pbb_{X_u^x}\right)\vert_{y=X_u^x}$ is an $\Fcal$-adapted process. Hence, we can apply the duality formula \cite{ProskeDiNunnoOksendal_MalliavinCalculus}[Corollary 4.4] and get
	\begin{align*}
		&\Ebb\left[ \int_0^T D_u\Phi(X_T^x) \partial_x b\left(u,y,\Pbb_{X_u^x}\right)\vert_{y=X_u^x} \int_0^u a(s) ds du \right] \\
		&\quad = \Ebb\left[ \Phi(X_T^x) \int_0^T \partial_x b\left(u,y,\Pbb_{X_u^x}\right)\vert_{y=X_u^x} \int_0^u a(s) ds dB_u \right].
	\end{align*}
	Putting all together provides representation \eqref{Delta} for $\Phi \in \Ccal_b^{1,1}(\Rbb)$. \par 
By standard arguments, we can now approximate $\Phi\in L^{2p}(\Rbb;\omega_T)$ by a smooth sequence $\lbrace \Phi_n \rbrace_{n\geq 1} \subset C_0^{\infty}(\Rbb)$ such that $\Phi_n \to \Phi$ in $L^{2p}(\Rbb; \omega_T)$ as $n\to\infty$. Define
	\begin{align*}
		u_n(x) &:= \Ebb \left[ \Phi_n(X_T^x) \right]\quad \text{and} \\
		\overline{u}(x) &:= \Ebb \left[ \Phi(X_T^x) \left( \int_0^T a(s) \partial_x X_s^x + \partial_x b(u,X_u^x,\Pbb_{X_u^x})\vert_{y=X_s^x} \int_0^u a(s) ds dB_u \right) \right].
	\end{align*}		
	First, we obtain that $\overline{u}$ is well-defined using Hölder's inequality, Itô's isometry and \Cref{boundsSolution}. Indeed,
	\begin{align}\label{wellDefined}
	\begin{split}
		|\overline{u}(x)| &\leq \Ebb \left[ \Phi(X_T^x)^2 \right]^{\frac{1}{2}} \Ebb \left[ \left( \int_0^T a(s) \partial_x X_s^x + \partial_x b(u,X_u^x,\Pbb_{X_u^x})\vert_{y=X_s^x} \int_0^u a(s) ds dB_u \right)^2 \right]^{\frac{1}{2}} \\
		&\leq \Ebb \left[ \Phi(B_T^x)^2 \Ecal\left( \int_0^T b(u,B_u^x,\rho_u^x)dB_u \right) \right]^{\frac{1}{2}}\\
		&\quad \times \Ebb \left[ \int_0^T \left(a(s) \partial_x X_s^x + \partial_x b(u,X_u^x,\Pbb_{X_u^x})\vert_{y=X_s^x} \int_0^u a(s) ds \right)^2 du \right]^{\frac{1}{2}} \\
		&\lesssim \Ebb \left[ \left| \Phi(B_T^x) \right|^{2p} \right]^{\frac{1}{2p}} < \infty,
	\end{split}
	\end{align}
	where the last inequality holds due to \Cref{pNormDerivative} and the proof of \Cref{driftInSobolev}. Similar to the proof of \Cref{driftInSobolev} it is left to show that $\langle u_n' - \overline{u}, \phi \rangle_U$ for any test-function $\phi \in \Ccal_0^{\infty}(U)$ as $n\to\infty$, where $U\subset K$ is an open set. Since the bounds in \eqref{wellDefined} hold for almost all $x\in U \subset K$, we get exactly in the same way that
	\begin{align*}
		|u'(x) - \overline{u}(x)| &\leq C(x) \Ebb \left[ \left| \Phi_n(B_T^x) - \Phi(B_T^x)\right|^{2p}\right]^{\frac{1}{2p}} \\
		&= C(x) \left( \int_{\Rbb} \frac{1}{\sqrt{2 \pi T}} \left| \Phi_n(y) - \Phi(y)\right|^{2p} e^{-\frac{(y-x)^2}{2T}} dy \right)^{\frac{1}{2p}} \\
		&\leq C(x) \left( \frac{e^{\frac{x^2}{2T}}}{\sqrt{2 \pi T}} \int_{\Rbb} \left| \Phi_n(y) - \Phi(y)\right|^{2p} e^{-\frac{y^2}{4T}} dy \right)^{\frac{1}{2p}} \\
		&= C(x) \left( \frac{e^{\frac{x^2}{2T}}}{\sqrt{2 \pi T}}\right)^{\frac{1}{2p}} \left\Vert \Phi_n - \Phi \right\Vert_{L^{2p}(\Rbb;\omega_T)},
	\end{align*}
	where $C(x)>0$ is bounded for almost every $x\in K$ and where we have used $$e^{-\frac{(y-x)^2}{2t}} = e^{-\frac{y^2}{4t}} e^{-\frac{(y-2x)^2}{4t}} e^{\frac{x^2}{2t}} \leq e^{-\frac{y^2}{4t}} e^{\frac{x^2}{2t}}.$$ Hence, for any open subset $U \subset K$, we get
	\begin{align*}
		\lim_{n\to\infty}  \langle u_n'(x) - \overline{u}(x), \phi \rangle_U = 0.
	\end{align*}
	Thus $u'= \overline{u}$ for almost every $x \in K$.
\end{proof}

\appendix
\section{Technical Results}
%

\begin{lemma}\label{pNorm}
	Let $b:[0,T] \times \Rbb \times \Pcal_1(\Rbb) \to \Rbb$ be a measurable function satisfying the linear growth condition \eqref{linearGrowth}. Furthermore, let $(\Omega, \Fcal, \Fbb, \Pbb, B, X^x)$ be a weak solution of \eqref{helpMcKeanVlasov}. Then, for $1\leq p < \infty$,  and every compact set $K\subset \Rbb$,
	\begin{align}\label{boundDrift}
		\sup_{x\in K} \Ebb\left[\sup_{t\in[0,T]} |b(t,X_t^x,\mu_t)|^p\right] < \infty.
	\end{align}
	In particular, $b(\cdot,X_{\cdot}^x,\mu_{\cdot}) \in L^p([0,T] \times \Omega)$, $1\leq p < \infty$. Furthermore,
	\begin{align}\label{exoFunctionalBounding}
		\sup_{x\in K} \Ebb\left[\sup_{t\in [0,T]} |X_t^x|^p\right] < \infty.
	\end{align}
\end{lemma}
\begin{proof}
	Note first that $\sup_{t\in[0,T]} \Kcal(\mu_t,\delta_0) dt$ is well-defined and finite. Indeed, since $\mu \in \Ccal([0,T];\Pcal_1(\Rbb))$ and $\Kcal(\cdot,\delta_0)$ is continuous, the supremum over $t\in[0,T]$ of $\Kcal(\mu_t, \delta_0)$ is attained. Furthermore, we can write
	\begin{align}\label{KantorovichToMean}
	\begin{split}
		\Kcal(\mu_t, \delta_0) &= \sup_{h\in \Lip_1} \left| \int_{\Rbb} h(y) \mu_t(dy) - h(0) \right| \leq \sup_{h\in \Lip_1} \int_{\Rbb} |h(y)-h(0)|  d\mu_t(dy) \\
		&\leq \int_{\Rbb} |y| \mu_t(dy) < \infty,
	\end{split}
	\end{align}
	where the last term is finite by the definition of $\Pcal_1(\Rbb)$. Therefore, we get due to the linear growth of $b$ that
	\begin{align*}
		|X_t^x| = \left| x + \int_0^t b(s,X_s^x, \mu_s) ds + B_t \right| \lesssim |x| + T + |B_t| + \int_0^t |X_s^x| ds.
	\end{align*}
	Thus, Grönwall's inequality yields that there exist constants $C_1$ and $C_2$ such that
	\begin{align}\label{eq:linearGrowthWithBrownian}
	\begin{split}
		|X_t^x| &\leq C_1 \left( 1+ |x| + \sup_{s\in[0,T]} |B_s^x| \right), \text{ and }\\
		|b(t,X_t^x,\mu_t)| &\leq C_2 \left( 1+ |x| + \sup_{s\in[0,T]} |B_s^x| \right).
	\end{split}
	\end{align}
	The boundedness of \eqref{boundDrift} is a direct consequence of \eqref{eq:linearGrowthWithBrownian} and Doob's maximal inequality.
\end{proof}

We define the complete probability space $(\Omega, \Fcal, \Qbb)$ carrying a Brownian motion $B$. In the following lemma we will prove the existence of an equivalent measure $\Pbb^{\mu}$ induced by the drift coefficient $b$.

\begin{lemma}\label{equivalentBrownianMeasurePbb}
	Let $b:[0,T] \times \Rbb \times \Pcal_1(\Rbb) \to \Rbb$ be a measurable function satisfying the linear growth condition \eqref{linearGrowth}. Then the Radon-Nikodym derivative
	\begin{align}\label{radonNikodymDerivative}
		\frac{d\Pbb^{\mu}}{d\Qbb} = \Ecal\left(\int_0^T b(s, B_s^x, \mu_s) dB_s \right)
	\end{align}
	is well-defined and yields a probability measure $\Pbb^{\mu} \approx \Qbb$.
	If $(\Omega, \Fcal, \Fbb, \Pbb^{\mu}, B^{\mu},X^x)$ is a weak solution of \eqref{helpMcKeanVlasov},  the Radon-Nikodym derivative
	\begin{align}\label{radonNikodymDerivative2}
		\frac{d\Qbb^{\mu}}{d\Pbb^{\mu}} = \Ecal\left( - \int_0^T b(s, X_s^x, \mu_s) dB^{\mu}_s \right)
	\end{align}
	is well-defined and yields a probability measure $\Qbb^{\mu}$ equivalent to $\Pbb^{\mu}$. Moreover, $(X_t^x)_{t\in[0,T]}$ is a $\Qbb^{\mu}$-Brownian motion starting in $x$.
\end{lemma}
\begin{proof} 
This is a direct consequence of Bene$\check{s}$' result and \eqref{eq:linearGrowthWithBrownian}.
\end{proof}

\begin{lemma} \label{boundMeasureChange}
	Let $b:[0,T] \times \Rbb \times \Pcal_1(\Rbb) \to \Rbb$ be a measurable function satisfying the linear growth condition \eqref{linearGrowth}. 
	Then, there exists an $\varepsilon > 0$ such that for any $\mu \in \Ccal([0,T];\Pcal_1(\Rbb))$,
	\begin{align}\label{eq:boundMeasureChange}
		\Ebb \left[ \Ecal \left( \int_0^T b(u,B_u^x, \mu_u) dB_u \right)^{1+\varepsilon} \right] < \infty.
	\end{align}
\end{lemma}
\begin{proof}
	First, we rewrite
	\begin{align*}
		 \Ebb &\left[ \Ecal \left( \int_0^T b(u,B_u^x, \mu_u) dB_u \right)^{1+\varepsilon} \right] \\
		 &= \Ebb \left[ \exp \left\lbrace \int_0^T (1+\varepsilon) b(u,B_u^x, \mu_u)dB_u - \frac{1}{2} \int_0^T (1+\varepsilon) \vert b(u,B_u^x, \mu_u) \vert^2 du \right\rbrace \right] \\
		 &= \Ebb \left[ \Ecal \left( \int_0^T (1+\varepsilon) b(u,B_u^x, \mu_u) dB_u \right) \exp \left\lbrace \frac{1}{2} \int_0^T \varepsilon(1+\varepsilon) \vert b(u,B_u^x, \mu_u)\vert^2 du \right\rbrace \right] \\
		 &= \Ebb \left[ \exp \left\lbrace \frac{1}{2} \int_0^T \varepsilon(1+\varepsilon) \vert b(u,X_u^{\varepsilon,x}, \mu_u)\vert^2 du \right\rbrace \right],
	\end{align*}
	where in the last step by Girsanov's theorem $X^{\varepsilon,x}$ denotes a weak solution of
	\begin{align*}
		dX_t^{\varepsilon ,x} &= (1+ \varepsilon) b(t,X_t^{\varepsilon ,x}, \mu_t) dt + dB_t, \quad X_0^{\varepsilon ,x} = x \in \Rbb, \quad t\in[0,T].
	\end{align*}
	Since $b$ satisfies the linear growth condition \eqref{linearGrowth}, we have that
	\begin{align*}
		\vert X_t^{\varepsilon,x} \vert &\leq \vert x \vert + (1+\varepsilon) \int_0^t \vert b(u,X_u^{\varepsilon,x}, \mu_u)\vert du + \vert B_t \vert \\
		&\leq \vert x \vert + C(1+ \varepsilon) \int_0^t (1+\vert X_u^{\varepsilon,x}\vert + \Kcal(\mu_u, \delta_0)) du + \vert B_t\vert.
	\end{align*}
	Therefore, Grönwall's inequality gives us
	\begin{align*}
		\vert X_t^{\varepsilon,x}\vert \leq (1+ \varepsilon) \left( T + \vert x \vert + \sup_{s\in[0,T]} \vert B_s \vert + \sup_{u\in[0,T]} \Kcal(\mu_u,\delta_0) \right) e^{C(1+\varepsilon)T},
	\end{align*}
	and thus, we can find a constant $C_{\varepsilon,\mu}$ depending on $\varepsilon$, $\mu$ and $T$ such that $\lim_{\varepsilon \to 0} C_{\varepsilon,\mu}$ exists, is finite, and
	\begin{align*}
		\vert b(t, X_t^{\varepsilon,x}, \mu_t) \vert \leq C_{\varepsilon,\mu} \left( 1 + \vert x \vert + \sup_{s\in[0,T]} \vert B_s \vert \right).
	\end{align*}
	Hence,
	\begin{align*}
		\Ebb &\left[ \exp \left\lbrace \frac{1}{2} \int_0^T \varepsilon (1+ \varepsilon) |b(u, X_u^{\varepsilon,x}, \mu_u)|^2 du \right\rbrace \right] \\
		&\leq \Ebb \left[ \exp \left\lbrace \frac{1}{2} T \varepsilon (1+ \varepsilon)C_{\varepsilon,\mu}^2 \left(1+ \vert x \vert + \sup_{s\in[0,T]} \vert B_s \vert \right)^2 \right\rbrace \right].
	\end{align*}
	Clearly, $\lim_{\varepsilon \to 0} \varepsilon(1+\varepsilon) C_{\varepsilon,\mu}^2 = 0$ and therefore we can choose $\varepsilon >0$ sufficiently small such that \eqref{eq:boundMeasureChange} holds.
\end{proof}

\begin{lemma}\label{boundsSolution}
	Let $b:[0,T] \times \Rbb \times \Pcal_1(\Rbb) \to \Rbb$ be a measurable function satisfying the linear growth condition \eqref{linearGrowth}. Furthermore, let $(\Omega, \Fcal, \Gbb, \Pbb, B, X^x)$ be a weak solution of the mean-field SDE \eqref{mainMcKeanVlasov}. Then,
	\begin{align}\label{boundDriftSolution}
		\vert b(t, X_t^x, \Pbb_{X_t^x})\vert \leq C \left(1+\vert x \vert + \sup_{s\in[0,T]} \vert B_s \vert \right)
	\end{align}
	for some constant $C>0$. Consequently, for any compact set $K\subset \Rbb$, and $1\leq p < \infty$, there exists $\varepsilon >0$ such that the following boundaries hold:
	\begin{align*}
		\sup_{x\in K} \Ebb \left[\sup_{t\in[0,T]} |b(t,X_t^x,\Pbb_{X_t^x})|^p\right] < \infty\\
		\sup_{x\in K} \sup_{t\in[0,T]} \Ebb \left[|X_t^x|^p\right] < \infty \\
		\sup_{x\in K} \Ebb \left[ \Ecal \left( \int_0^T b(u,B_u^x, \Pbb_{X_u^x}) dB_u \right)^{1+\varepsilon} \right] < \infty
	\end{align*}
\end{lemma}
\begin{proof}
	Due to the proofs of \Cref{pNorm} and \Cref{boundMeasureChange}, it suffices to show \eqref{boundDriftSolution}. Note first that $\Kcal(\Pbb_{X_t^x},\delta_0) \leq \Ebb[|X_t^x|]$ for every $t\in [0,T]$ by \eqref{KantorovichToMean}. Hence, it is enough to show that $\Ebb[|X_t^x|] \leq C( 1 + |x| )$ for every $t\in [0,T]$ and some constant $C>0$. Since $(X_t^x)_{t\in[0,T]}$ is a weak solution of \eqref{mainMcKeanVlasov} and $b$ fulfills the linear growth condition \eqref{linearGrowth}, we get
	\begin{align*}
		\Ebb[|X_t^x|] &\lesssim |x| + \int_0^t 1 + \Ebb[|X_s^x|] + \Kcal(\Pbb_{X_s^x},\delta_0) ds + \Ebb[|B_t|] \lesssim  1 + |x| + \int_0^t \Ebb[|X_s^x|] ds.
	\end{align*}
	Consequently $\Ebb[|X_t^x|] \leq C( 1 + |x| )$ by Grönwall's inequality which concludes the proof.
\end{proof}

\begin{lemma}\label{boundLocalTime}
	Suppose the drift coefficient $b:[0,T]\times \Rbb \times \Pcal_1(\Rbb) \to \Rbb$ is in the decomposable form \eqref{formDrift} and there exists a constant $C>0$ such that $(\mu \mapsto b(t,y,\mu)) \in \Lip_C(\Pcal_1(\Rbb))$ for every $t\in[0,T]$ and $y\in\Rbb$. Let $(X_t^x)_{t\in[0,T]}$ be the unique strong solution of \eqref{mainMcKeanVlasov}. Furthermore, $\lbrace b_n \rbrace_{n\geq 1}$ is the approximating sequence of $b$ as defined in \eqref{approximatingDrift} and $(X_t^{n,x})_{t\in[0,T]}$, $n\geq 1$, the corresponding unique strong solutions of \eqref{approximatingMFSDEExp}. Then, for all $\lambda \in \Rbb$ and any compact subset $K \subset \Rbb$,
	\begin{align*}
		\sup_{n\geq 0} \sup_{s,t\in[0,T]} \sup_{x\in K} \Ebb \left[ \exp \left\lbrace - \lambda \int_s^t \int_{\Rbb} b_n\left(s,y,\Pbb_{X_s^{n,x}}\right) L^{B^x}(ds,dy) \right\rbrace \right] < \infty.
	\end{align*}
\end{lemma}
\begin{proof}
	Recall that $b_n$ can be decomposed into $b_n = \btilde_n + \bhat$ for all $n\geq 0$. Here $\btilde_n$ is uniformly bounded in $n\geq 0$. Hence, by \cite{MeyerBrandisBanosDuedahlProske_ComputingDeltas}[Lemma A.2]
		\begin{align*}
			\sup_{n\geq 0} \sup_{s,t\in[0,T]} \sup_{x\in K} \Ebb \left[ \exp \left\lbrace - \lambda \int_s^t \int_{\Rbb} \btilde_n\left(s,y,\Pbb_{X_s^{n,x}}\right) L^{B^x}(ds,dy) \right\rbrace \right] < \infty.
		\end{align*}
		Moreover, $\Vert \partial_2 \bhat \Vert_{\infty} < \infty$ by definition. Consequently,
		\begin{align*}
			\sup_{n\geq 0} \sup_{s,t\in[0,T]} \sup_{x\in K} &\Ebb \left[ \exp \left\lbrace - \lambda \int_s^t \int_{\Rbb} \bhat\left(s,y,\Pbb_{X_s^{n,x}}\right) L^{B^x}(ds,dy) \right\rbrace \right] \\
			&= \sup_{n\geq 0} \sup_{s,t\in[0,T]} \sup_{x\in K} \Ebb \left[ \exp \left\lbrace \lambda \int_s^t \partial_2 \bhat \left(s,B_s^x,\Pbb_{X_s^{n,x}}\right) ds \right\rbrace \right] < \infty.
		\end{align*}
\end{proof}

\begin{lemma}\label{uniformConvergenceEcal}
	Suppose the drift coefficient $b:[0,T]\times \Rbb \times \Pcal_1(\Rbb) \to \Rbb$ is in the decomposable form \eqref{formDrift} and there exists a constant $C>0$ such that $(\mu \mapsto b(t,y,\mu)) \in \Lip_C(\Pcal_1(\Rbb))$ for every $t\in[0,T]$ and $y\in\Rbb$. Let $(X_t^x)_{t\in[0,T]}$ be the unique strong solution of \eqref{mainMcKeanVlasov}. Furthermore, $\lbrace b_n \rbrace_{n\geq 1}$ is the approximating sequence of $b$ as defined in \eqref{approximatingDrift} and $(X_t^{n,x})_{t\in[0,T]}$, $n\geq 1$, the corresponding unique strong solutions of \eqref{approximatingMFSDEExp}. Then for any compact subset $K \subset \Rbb$ and $q:= \frac{2(1+\varepsilon)}{2+\varepsilon}$, $\varepsilon>0$ sufficiently small with regard to \Cref{boundsSolution},
	\begin{align*}
		\sup_{x\in K} \Ebb \left[ \left| \Ecal \left( \int_0^T b_n(t,B_t^x,\Pbb_{X_t^{n,x}}) dB_t \right) - \Ecal \left( \int_0^T b(t,B_t^x,\Pbb_{X_t^x}) dB_t \right) \right|^q \right]^{\frac{1}{q}} \xrightarrow[n\to\infty]{} 0.
	\end{align*}
\end{lemma}
\begin{proof}
	For the sake of readability we use the abbreviation $\bfr_n(X_t^{k,x}) = b_n(t,B_t^x,\Pbb_{X_t^{k,x}})$ for $n,k \geq 0$. First using inequality \eqref{exponentialInequality}, \Cref{boundsSolution} and Burkholder-Davis-Gundy's inequality yields
	\begin{align*}
		A_n&(T,x):= \Ebb \left[ \left| \Ecal \left( \int_0^T \bfr_n(X_t^{n,x}) dB_t \right) - \Ecal \left( \int_0^T \bfr(X_t^x) dB_t \right) \right|^q \right]^{\frac{1}{q}} \\
		&\leq \Ebb\left[ \left| \int_0^T \bfr_n(X_t^{n,x}) - \bfr(X_t^x) dB_t + \frac{1}{2} \int_0^T \bfr_n(X_t^{n,x})^2 - \bfr(X_t^x)^2 dt \right|^q \right.\\
		&\quad \left.\left( \Ecal \left( \int_0^T \bfr_n(X_t^{n,x}) dB_t \right) + \Ecal \left( \int_0^T \bfr(X_t^x) dB_t \right) \right)^q \right]^{\frac{1}{q}} \\
		&\lesssim \Ebb\left[ \left| \int_0^T \left( \bfr_n(X_t^{n,x}) - \bfr(X_t^x)\right)^2 dt \right|^{\frac{p}{2}} \right]^{\frac{1}{p}}  + \Ebb\left[ \left| \int_0^T \bfr_n(X_t^{n,x})^2 - \bfr(X_t^x)^2 dt \right|^p \right]^{\frac{1}{p}},
	\end{align*}
	where $p:= \frac{1+\varepsilon}{\varepsilon}$. Due to its definition $b_n$ is of linear growth uniformly in $n\geq 0$ and thus we get with \Cref{boundsSolution} that
	\begin{align*}
		\Ebb&\left[ \left| \bfr_n(X_t^{n,x})^2 - \bfr(X_t^x)^2 \right|^p \right]^{\frac{1}{p}} \lesssim \Ebb\left[ \left| \bfr_n(X_t^{n,x}) - \bfr(X_t^x) \right|^{2p} \right]^{\frac{1}{2p}}
	\end{align*}
	and by Minkowski's integral as well as Cauchy-Schwarz' inequality, we have
	\begin{align*}
		A_n&(T,x)\\
		&\lesssim \left( \int_0^T \Ebb\left[ \left| \bfr_n(X_t^{n,x}) - \bfr(X_t^x)\right|^{2p} \right]^{\frac{2}{2p}} dt \right)^{\frac{1}{2}} + \int_0^T \Ebb\left[ \left| \bfr_n(X_t^{n,x}) - \bfr(X_t^x) \right|^{2p} \right]^{\frac{1}{2p}} dt \\
		&\lesssim \left( \int_0^T \Ebb\left[ \left| \bfr_n(X_t^{n,x}) - \bfr(X_t^x)\right|^{2p} \right]^{\frac{2}{2p}} dt \right)^{\frac{1}{2}}.
	\end{align*}
	Using the triangle inequality and $(\mu \mapsto b(t,y,\mu)) \in \Lip_C(\Pcal_1(\Rbb))$ for every $t\in[0,T]$ and $y\in\Rbb$ yields
	\begin{align*}
		\Ebb&\left[ \left| \bfr_n(X_t^{n,x}) - \bfr(X_t^x) \right|^{2p} \right]^{\frac{1}{2p}} \\
		&\leq \Ebb\left[ \left| \bfr_n(X_t^{n,x}) - \bfr_n(X_t^x) \right|^{2p} \right]^{\frac{1}{2p}} + \Ebb\left[ \left| \bfr_n(X_t^x) - \bfr(X_t^x) \right|^{2p} \right]^{\frac{1}{2p}} \\
		&\leq C \Kcal\left( \Pbb_{X_t^{n,x}}, \Pbb_{X_t^x} \right) + D_n(t,x) \leq C \Ebb\left[ \left| X_t^{n,x} - X_t^x \right| \right] + D_n(t,x),
	\end{align*}
	where $D_n(t,x) := \Ebb\left[ \left| \bfr_n(X_t^x) - \bfr(X_t^x) \right|^{2p} \right]^{\frac{1}{2p}}$, $t\in[0,T]$. With Girsanov's Theorem and Jensen's inequality we get
	\begin{align*}
		\Ebb\left[\left|  X_t^{n,x} - X_t^x \right| \right] &= \Ebb \left[ |B_t^x| \left| \Ecal \left( \int_0^t \bfr_n(X_s^{n,x}) dB_s \right) - \Ecal \left( \int_0^t \bfr(X_s^x) dB_s \right) \right|\right]\\
		&\lesssim \Ebb \left[ \left| \Ecal \left( \int_0^t \bfr_n(X_s^{n,x}) dB_s \right) - \Ecal \left( \int_0^t \bfr(X_s^x) dB_s \right) \right|^q \right]^{\frac{1}{q}} = A_n(t,x).
	\end{align*}
	Consequently, $A_n(T,x) \lesssim \left( \int_0^T (A_n(t,x) + D_n(t,x))^2 dt \right)^{\frac{1}{2}}$ and therefore
	\begin{align*}
		A_n^2(T,x) \lesssim \int_0^T A_n^2(t,x) dt + \int_0^T D_n^2(t,x) dt.
	\end{align*}
	Hence, we get with Grönwall's inequality
	\begin{align*}
		A_n^2(T,x) \leq C \int_0^T D_n^2(t,x) dt,
	\end{align*}
	for some constants $C >0$ independent of $x\in K$, $n\geq 0$ and $t\in[0,T]$ and as a consequence it suffices to show
	\begin{align}\label{Dconvergence}
		\sup_{x\in K} \int_0^T D_n^2(t,x) dt \xrightarrow[n\to\infty]{} 0.
	\end{align}
	Note first
	\begin{align*}
		D_n^2(t,x) &= \Ebb\left[ \left| b_n\left(t,B_t^x,\Pbb_{X_t^x}\right) - b\left(t,B_t^x,\Pbb_{X_t^x}\right) \right|^{2p} \right]^{\frac{2}{2p}} \\
		&= \left( \int_{\Rbb} \left| b_n\left(t,y,\Pbb_{X_t^x}\right) - b\left(t,y,\Pbb_{X_t^x}\right) \right|^{2p} \frac{1}{\sqrt{2\pi t}} e^{-\frac{(y-x)^2}{2t}} dy \right)^{\frac{2}{2p}} \\
		&\leq e^{\frac{x^2}{2pt}} \left( \int_{\Rbb} \left| b_n\left(t,y,\Pbb_{X_t^x}\right) - b\left(t,y,\Pbb_{X_t^x}\right) \right|^{2p} \frac{1}{\sqrt{2\pi t}} e^{-\frac{y^2}{4t}} dy\right)^{\frac{2}{2p}},
	\end{align*}
	where we have used $e^{-\frac{(y-x)^2}{2t}} = e^{-\frac{y^2}{4t}} e^{-\frac{(y-2x)^2}{4t}} e^{\frac{x^2}{2t}} \leq e^{-\frac{y^2}{4t}} e^{\frac{x^2}{2t}}$.  Furthermore, by \eqref{measureLipschitz} and \Cref{convergenceApproximating}, $\Pbb_{X_t^{\cdot}}$ is continuous for all $t\in[0,T]$ and thus, $\Pbb_{X_t^K} := \lbrace \Pbb_{X_t^x}: x\in K \rbrace \subset \Rbb$ is compact. Therefore due to the definition of the approximating sequence
	\begin{align*}
		\sup_{x\in K} \left| b_n\left(t,y,\Pbb_{X_t^x}\right) - b\left(t,y,\Pbb_{X_t^x}\right) \right| = \sup_{z \in \Pbb_{X_t^K}} \left| b_n(t,y,z) - b(t,y,z) \right| \xrightarrow[n\to\infty]{} 0,
	\end{align*}
	and hence $D_n^2(t,x)$ converges to $0$ uniformly in $x\in K$. Consequently, $\int_0^T D_n^2(t,x) dt$ converges uniformly to $0$ by \Cref{boundsSolution} and dominated convergence, which proves the result.
\end{proof}

\begin{lemma}\label{uniformConvergenceLocalTime}
	Suppose the drift coefficient $b:[0,T]\times \Rbb \times \Pcal_1(\Rbb) \to \Rbb$ is in the decomposable form \eqref{formDrift} and there exists a constant $C>0$ such that $(\mu \mapsto b(t,y,\mu)) \in \Lip_C(\Pcal_1(\Rbb))$ for every $t\in[0,T]$ and $y\in\Rbb$. Let $(X_t^x)_{t\in[0,T]}$ be the unique strong solution of \eqref{mainMcKeanVlasov}. Furthermore, $\lbrace b_n \rbrace_{n\geq 1}$ is the approximating sequence of $b$ as defined in \eqref{approximatingDrift} and $(X_t^{n,x})_{t\in[0,T]}$, $n\geq 1$, the corresponding unique strong solutions of \eqref{approximatingMFSDEExp}. Then for any compact subset $K \subset \Rbb$, $s,t\in[0,T]$, $s\leq t$ and $p\geq 1$,
	\begin{align*}
		\Ebb \left[ \left| \exp \left\lbrace - \int_s^t \int_{\Rbb} \bfr_n(u,y) L^{B^x}(du,dy) \right\rbrace - \exp \left\lbrace - \int_s^t \int_{\Rbb} \bfr(u,y) L^{B^x}(du,dy) \right\rbrace \right|^p \right]^{\frac{1}{p}},
	\end{align*}
	where $\bfr_n(u,y) := b_n\left(u,y,\Pbb_{X_u^{n,x}}\right)$ for all $n\geq 0$, converges uniformly in $x\in K$ to $0$ as $n$ goes to infinity.
\end{lemma}
\begin{proof}
	 We first use inequality \eqref{exponentialInequality} to obtain with \Cref{boundLocalTime}
	\begin{align*}
		\Ebb&\left[ \left| \exp \left\lbrace - \int_s^t \int_{\Rbb} \bfr_n(u,y) L^{B^x}(du,dy) \right\rbrace - \exp \left\lbrace - \int_s^t \int_{\Rbb} \bfr(u,y) L^{B^x}(du,dy) \right\rbrace \right|^p \right]^{\frac{1}{p}} \\
		& \leq \Ebb\left[ \left| \int_s^t \int_{\Rbb} \bfr_n(u,y) L^{B^x}(du,dy) - \int_s^t \int_{\Rbb} \bfr(u,y) L^{B^x}(du,dy)\right|^p \right. \\
		&\times \left.\left(\exp \left\lbrace - \int_s^t \int_{\Rbb} \bfr_n(u,y) L^{B^x}(du,dy) \right\rbrace + \exp \left\lbrace - \int_s^t \int_{\Rbb} \bfr(u,y) L^{B^x}(du,dy) \right\rbrace \right)^{p} \right]^{\frac{1}{p}} \\
		&\lesssim \Ebb\left[ \left| \int_s^t \int_{\Rbb} \bfr_n(u,y) L^{B^x}(du,dy) - \int_s^t \int_{\Rbb} \bfr(u,y) L^{B^x}(du,dy)\right|^{2p} \right]^{\frac{1}{2p}}.
	\end{align*}
	We define the time-reversed Brownian motion $\hat{B}_t:= B_{T-t}$, $t\in[0,T]$, and the Brownian motion $W_t$, $t\in [0,T]$, with respect to the natural filtration of $\hat{B}$. By \cite{MeyerBrandisBanosDuedahlProske_ComputingDeltas}[Theorem 2.10], Burkholder-Davis-Gundy's inequality and Cauchy-Schwarz' inequality
	\begin{align*}
		\Ebb&\left[ \left|\int_s^t \int_{\Rbb} \bfr_n(u,y) - \bfr(u,y) L^{B^x}(du,dy)\right|^{2p}\right]^{\frac{1}{2p}} \\
		&= \Ebb\left[ \left|\int_s^t \bfr_n(u,B_u^x) - \bfr(u,B_u^x) dB_u + \int_{T-t}^{T-s} \bfr_n(T-u,\hat{B}_u^x) - \bfr(T-u,\hat{B}_u^x) dW_u\right.\right. \\
		&\quad \left.\left.- \int_{T-t}^{T-s} \left( \bfr_n(T-u,\hat{B}_u^x) - \bfr(T-u,\hat{B}_u^x) \right) \frac{\hat{B}_u}{T-u} du\right|^{2p}\right]^{\frac{1}{2p}} \\
		&\lesssim \Ebb\left[ \left(\int_s^t \left( \bfr_n(u,B_u^x) - \bfr(u,B_u^x) \right)^2 du\right)^{p}\right]^{\frac{1}{2p}} \\
		&\quad + \Ebb\left[ \left( \int_{T-t}^{T-s} \left( \bfr_n(T-u,\hat{B}_u^x) - \bfr(T-u,\hat{B}_u^x) \right)^2 du\right)^{p}\right]^{\frac{1}{2p}} \\
		&\quad + \int_{T-t}^{T-s} \left\Vert \bfr_n(T-u,\hat{B}_u^x) - \bfr(T-u,\hat{B}_u^x) \right\Vert_{L^{4p}(\Omega)} \left\Vert \frac{\hat{B}_u}{T-u}\right\Vert_{L^{4p}(\Omega)} du.
	\end{align*}
	Similar to the proof of \Cref{uniformConvergenceEcal} one obtains the result.
\end{proof}

\section{Hida spaces}
 In order to prove \Cref{convergenceApproximating}, we need the definition of the Hida test function and distribution space (cf. \cite{ProskeDiNunnoOksendal_MalliavinCalculus}[Definition 5.6]). Furthermore we state the central theorem used in the proof of \Cref{convergenceApproximating}, followed by a further helpful criterion for relative compactness using modulus of continuity.

\begin{definition}\label{HidaSpaces}
Let $\mathcal{I}$ be the set of all finite multi-indices and $\lbrace
H_{\alpha} \rbrace_{\alpha \in \mathcal{I}}$ be an orthogonal basis of the Hilbert space $L^2(\Omega)$ defined by
\begin{align*}
	H_{\alpha}(\omega) := \prod_{j=1}^m h_{\alpha_j}\left( \int_{\Rbb} e_j(t)dW_t (\omega)\right),
\end{align*}
where $h_n$ is the $n$-th hermitian polynomial, $e_n$ the $n$-th hermitian function and $W$ a standard Brownian motion. Furthermore, we define for every $\alpha =(\alpha_1,\dots \alpha_m)\in \mathcal{I}$,
\begin{align*}
	(2\Nbb)^{\alpha} := \prod_{j=1}^m (2j)^{\alpha_j}.
\end{align*}
\begin{enumerate}[(i)]
\item We define the Hida test function Space $\Scal$ as
\begin{align*}
	\Scal := \left\lbrace \phi= \sum_{\alpha \in \mathcal{I}} a_{\alpha} H_{\alpha} \in L^2(\Omega) : \Vert \phi \Vert_k < \infty \text{ } \forall k\in \Rbb \right\rbrace,
\end{align*}
where the norm $\Vert \cdot \Vert_k$ is defined by
\begin{align*}
	\Vert \phi \Vert_k := \sqrt{\sum_{\alpha \in \mathcal{I}} \alpha! a^2_{\alpha} (2\Nbb)^{\alpha k}}.
\end{align*}
Here, $\Scal$ is equipped with the projective topology.
\item The Hida distribution space $\Scal^*$ is defined by
\begin{align*}
	\Scal^* := \left\lbrace \phi= \sum_{\alpha \in \mathcal{I}} a_{\alpha} H_{\alpha} \in L^2(\Omega) : \exists k \in \Rbb \text{ s.t. } \Vert \phi \Vert_{-k} < \infty \right\rbrace,
\end{align*}
where the norm $\Vert \cdot \Vert_{-k}$ is defined by
\begin{align*}
	\Vert \phi \Vert_{-k} := \sqrt{\sum_{\alpha \in \mathcal{I}} \alpha! a^2_{\alpha} (2\Nbb)^{-\alpha k}}.
\end{align*}
Here, $\Scal^*$ is equipped with the inductive topology.
\end{enumerate}
\end{definition}

\begin{theorem}[Mitoma]\label{mitoma}
	The following statements are equivalent:
	\begin{enumerate}[(i)]
	\item $\Acal$ is relatively compact in $\Ccal([0,T]; \Scal^*)$,
	\item For any $\phi \in \Scal$, $\lbrace f(\cdot)[\phi]: f\in \Acal \rbrace$ is relatively compact in $\Ccal([0,T];\Rbb)$.
	\end{enumerate}
\end{theorem}
\begin{proof}
	\cite{KallianpurXiong_SDEsInInfiniteDimensionalSpaces}[Theorem 2.4.4]
\end{proof}

\begin{lemma}\label{relativeCompactModulus}
	Let $(X,\Vert \cdot \Vert)$ be a Banach space. Then $\Acal \subset \Ccal([0,T],X)$ is relatively compact if and only if 
	\begin{align*}
		\sup_{f \in \Acal} \sup \lbrace \Vert f(t)-f(s) \Vert: s,t \in [0,T], |t-s|<\delta \rbrace \rightarrow 0,
	\end{align*}
	as $\delta \to 0$.
\end{lemma}
\begin{proof}
	\cite{KallianpurXiong_SDEsInInfiniteDimensionalSpaces}[Theorem 2.4.3]
\end{proof}

\bibliography{literature}
\bibliographystyle{abbrv}

\end{document}